\documentclass[12pt, a4paper, notitlepage]{report}
\usepackage[top=1.2in, bottom=1.2in, left=0.8in, right=0.8in]{geometry}
\usepackage{textcomp} 
\usepackage{amsmath}
\usepackage{amsthm} 
\usepackage{mathtools}
\usepackage{amssymb}
\usepackage{amsfonts}
\usepackage{titling}
\usepackage{epigraph} 
\epigraphnoindent

\usepackage[compact]{titlesec} 

\titleformat{\chapter}
  {\normalfont\huge\bfseries}{\thechapter}{1em}{}

\titlespacing*{\chapter}{0cm}{-\topskip}{10pt}[0pt]

\usepackage{xcolor}
\usepackage{pdflscape} 
\usepackage{cancel}    
\usepackage{cite}
\usepackage[round]{natbib} 
\usepackage[normalem]{ulem}

\usepackage[bottom]{footmisc} 
\usepackage{tikz}
\usetikzlibrary{automata, chains}
\usetikzlibrary{positioning}  
\usetikzlibrary{arrows, calc}  
\tikzset{node distance=1.5cm, 
         every state/.style={
           semithick, fill=gray!3,
            minimum size=1.31cm},
         initial text={},     
         double distance=2pt, 
         every edge/.style={  
         draw,
           ->,>=stealth',     
           auto,
           semithick}}

\usepackage{bm} 

\newenvironment{dedication} 
  {\clearpage           	
   \thispagestyle{empty}	
   \vspace*{\stretch{1}}	 
   \itshape            	 	
   \centering  				
  }
  {\par 					
   \vspace{\stretch{3}} 	
   \clearpage
  }
  
\newenvironment{motto} 
  {\clearpage           	
   \thispagestyle{empty}	
   \vspace*{\stretch{1}} 
   \centering  				
  }
  {\par 					
   \vspace{\stretch{3}} 	
   \clearpage           	
  }  
  
\usepackage[ruled,vlined,linesnumbered]{algorithm2e} 
\SetKwRepeat{Do}{do}{while}
\usepackage{booktabs}
\usepackage{array} 

\usepackage{caption}
\usepackage{xfrac}

\DeclareRobustCommand{\stirlingtwo}{ \genfrac\{\}{0pt}{} } 

\newtheorem{theorem}{Theorem}
\newtheorem*{theorem*}{Theorem}
\newtheorem{proposition}{Proposition}
\newtheorem*{proposition*}{Proposition}
\newtheorem{lemma}{Lemma}
\newtheorem{corollary}{Corollary}
\newtheorem{definition}{Formulation}

\DeclareMathOperator{\diag}{diag}
\DeclareMathOperator{\E}{\mathbb{E}}
\DeclareMathOperator{\V}{\mathbb{V}}

\DeclareMathOperator{\ceil}{ceil}

\usepackage{mathtools} 
\def\multiset#1#2{\ensuremath{\left(\kern-.3em\left(\genfrac{}{}{0pt}{}{#1}{#2}\right)\kern-.3em\right)}}
\def\multichoose#1#2{\ensuremath{\left(\kern-.3em\left(\genfrac{}{}{0pt}{}{#1}{#2}\right)\kern-.3em\right)}}

\usepackage{hyperref}
\usepackage{cleveref}

\providecommand{\keywords}[1] 
{
  \small	
  \textbf{\textit{Keywords  }} #1
}

\usepackage{tikz}

\title{	A Markovian Perspective on the \\
			 Classical Occupancy Problem \\
			\large	with a Generalization to Pure Birth Processes}

\author{Jim van Mechelen}

\date{July, 2020}

\begin{document}

\begin{titlingpage}
\maketitle

\abstract{We study the classical occupancy problem from the viewpoint of its embedding Markov chain. We derive new expressions for the probability mass function and (complementary) distribution function in generalized form. Furthermore, we derive a completely novel sparse bidiagonal system of recursion relations for the (complementary) distribution function and provide its efficient matrix implementation.

Importantly, we generalize these results to the entire class of discrete-time pure birth processes. 
} \newline \newline


\end{titlingpage}

\clearpage
\thispagestyle{empty}

\vspace{10cm}

Master Thesis in Econometrics Presented for the Degree of Master of Science \\

\begin{tabular}{rl}
  \textit{by} \\
  Candidate:	& Jim van Mechelen  \\ 
  \\
  \textit{before}	&			\\
  Supervisor:	& Dr. M. van de Velden \\ 
  Assessor:		& Dr. A.J. Koning \\
  \\
  				& Department of Econometrics, Erasmus University Rotterdam \\
				& \\ 	
				& \\ 
	 			& Publicly defended on 24 September 2020.
	 			  				
\end{tabular}

\null\vfill
\noindent
$^{\tiny\copyright}$ 2020 - present. All rights reserved. Jim van Mechelen ({vmechelen\{at\}gmail[dot]com}). \\
Copy for \href{https://arxiv.org/}{arXiv}.
\\ No part of this publication may be (re-)produced, stored in a retrieval system, or transmitted in any form or by any means, electronic, mechanical, photocopying, recording, or otherwise, without written permission of the author. For information regarding permissions, \emph{contact the author}.

\begin{dedication}
To my mother \\
Sari de Koning  
\end{dedication}

\begin{motto}
Luctor et Emergo   
\end{motto}

\newpage

\tableofcontents

\newpage

\chapter{Introduction} 
All empirical science starts with observation. Statistics, in particular, is the science concerned with assigning degrees of certainty or uncertainty to statements generalizing from a small number of observations, the sample, to all possible observations, the population.

Sampling comes in two main flavors: \emph{with} and \emph{without} replacement. In this thesis we shall concentrate on the first. For one, because sampling without replacement is well understood and holds few secrets. 

Sampling with replacement has received renewed interest in recent decades following the rising popularity -- some would say revolution -- of non-parametric statistics and machine learning. Two prime examples are \cite{efron1979bootstrap}'s bootstrap, and \cite{breiman2001randomforest}'s random forest. Sampling with replacement is a pivotal element in both Efron's bootstrap procedure as well as Breiman's random forest. Conceptually, sampling with replacement is not more complicated than sampling without replacement, however its probabilistic properties are quite more involved. 

In this thesis we restrict our attention to an essential property of a sample obtained by sampling with replacement, namely: the number of distinct elements that the sample contains. This is the \emph{classical occupancy problem}. We concern ourselves only with the exact distribution and not with its approximation or asymptotics.

The problem is a true classic of probability theory and traces its origins to \cite{demoivre1711mensura} who described it in terms of casting an $n$-faced die $t$ times and asking how many, and with which probability, $k$ distinct faces would be observed. Both Euler and Laplace studied generalizations of the problem \citep{todhunter1865history}.

The classical occupancy problem gets its current name from a twentieth century rephrasing in terms of an urn model (\citet{feller1968ed3, kotz_johnson1977urn_models}; or seq. \Cref{sec:cocc_problem_definition}).

The classical occupancy problem can as well be motivated from its occurrence in natural phenomena and applications. \citet[Ch. 1.2]{feller1968ed3} lists sixteen variants of the problem ranging from the educational birthday problem to applications in physics, chemistry, biology and the distribution of the number of misprints in a book. In fact, the classical occupancy problem makes continual appearances in Feller's classic text. Some further examples include the mathematical analysis of card shuffling, estimating population size in ecology, public health surveillance, and the likelihood of collisions in lotteries \citep{aldous_diaconis1986shuffling, ziegler2014shuffling_cards, chao1984nonparametric, barabesi2011theta_param_est, williamson2009classic_occ}. Finally, in the context of a slightly more general model, \cite{harkness1969occupancy} lists five applications ranging from air battle theory to the spread of infectious disease.

Of theoretical interest is that, even though the order of the sampled elements is discarded by the statistician, sampling without replacement does not result in a uniformly random selection from the universe of multisets. This situation is reminiscent of a similar difference between the Maxwell-Boltzmann and Bose-Einstein statistics \citep{bose1924plancksgesetz, feller1968ed3}.

Notwithstanding these diverse manifestations and applications, the classical occupancy distribution does not appear to be a ``household name" among statisticians. This thesis serves in part to ameliorate that situation.

A focal point in this thesis is that the classical occupancy problem permits a Markov chain embedding. Some authors were aware of the Markovian nature\footnote{By Markovian nature we mean its memoryless property, that is: the continuation of the process depends only on the current state.} of the classical occupancy problem \citep[e.g.][]{polya1930coupon, renyi1962three_proofs, feller1968ed3}, but did not pursue this viewpoint further.
The one notable exception is \cite{uppuluri1971occupancy}. However, we believe that these authors did not exhaust the Markovian vantage point to its full potential.\footnote{For example: lemmas are stated without proof, the distribution function is not treated, etc.} We therefore pose our main research question as follows: 
\begin{quote}
\emph{What can we learn about the classical occupancy problem from a purely Markovian point of view?}
\end{quote}

\paragraph{More Generally.}
Since the embedding Markov chain defines a discrete-time pure birth process, it is natural to ask if, and to what extent, results generalize to the entire class of discrete-time pure birth processes.

Pure birth processes occupy an important role in probability theory. They are widely applied in areas such as biology, and reliability theory, as well as the theory of phase-type distributions, which are a major class of probability distributions \citep[cf.][]{neuts1975PH_distributions_liber_amicorum, ocinneide1989PHrepresentation}. 

Within the phase-type literature pure birth processes play a prominent role as \cite{cumani1982PHcanonical} showed that phase-type distributions with acyclic generator matrices can be reduced to a canonical pure birth representation. The phase-type literature is primarily concerned with phase duration and to a lesser extent with the distribution of the process over its states given a fixed amount of time -- even though these distributions are dual to each other, as we will see in the sequel. 

\subsection*{Thesis Outline}
We now provide an outline for the remainder. Chapters \ref{chpt:pbp_in_disc_time} and \ref{chpt:special_case} form the backbone of this thesis. Preceding these chapters we included a \nameref{chpt:preliminaries} chapter in which we describe some of the main tools that we use throughout this thesis. These tools are the complete homogeneous symmetric polynomial, the $r$-Stirling numbers of the second kind as a special case, and their representation in terms of finite differences.

\paragraph{Pure Birth Process.}
\Cref{chpt:pbp_in_disc_time} starts by introducing the general discrete-time pure birth process. In \Cref{sec:pbp_first_hit_times} we treat the duality between phase duration (first hitting times) and the probability distribution over process states. 

The complete homogeneous symmetric polynomial is extensively used in the \nameref{sec:pbp_results} section (\Cref{sec:pbp_results}) when we derive the exact form of the probability mass function and distribution function over the process states in Theorems \ref{thm:pbp_general_pmf} and \ref{thm:pbp_general_ccdf}, respectively. These are the most general results that we derive in this thesis.
Those theorems are then refined in \Cref{sec:pbp_results_excact_distri_distinct} for the pure birth process having distinct transition probabilities.

Finally, we derive a completely novel system of recurrence relations for the distribution function in \Cref{sec:pbp_results_ccdf_recur}, and provide its matrix form in \Cref{prop:pbp_ccdf_recur_matrixform}. We conclude with the description of a simple algorithm to simulate a pure birth process in \Cref{sec:pbp_simulation}.

\paragraph{Classical Occupancy.}
\Cref{chpt:special_case} centers completely on the classical occupancy problem and begins with a detailed introduction that explains how the classical occupancy problem is connected to sampling with replacement.

We collect various equivalent problem formulations in \Cref{sec:cocc_problem_definition}. The Markov chain embedding is introduced formally in the first part of \Cref{sec:cocc_markov_chain}.

Naturally, all results for the classical occupancy problem follow from substituting its specific parametrization -- in terms of its transition probabilities -- into the propositions of \Cref{chpt:pbp_in_disc_time}. However, we want to present an alternative route to the respective results by means of a matrix-algebraic derivation based on the eigendecomposition of the transition matrix. We derive this eigendecomposition in \Cref{sec:eigendecomposition}. 

Having the eigendecomposition in hand, we derive our results specific to the classical occupancy problem in \Cref{sec:cocc_results}. We do so in explicit form, in terms of $r$-Stirling numbers, and as finite differences. Of particular interest are the initial-state conditioned versions for the probability mass function and distribution function, these are respectively \Cref{prop:transition_probs_classical_occupancy} (Sec. \ref{sec:other_trans_probs}) and \Cref{prop:cdf_from_eigendecomp} (Sec. \ref{sec:cocc_distri_func}).

Finally, in \Cref{sec:rand_occ_model} we make a short digression into \cite{harkness1969occupancy}'s \emph{randomized} occupancy model, and answer a question posed by Harkness affirmatively with expression \eqref{eq:cond_pmf_randomized_occ_harkness}.\\

\noindent We end this thesis with a \nameref{chpt:conclusion_discussion}. The \nameref{chpt:apx} contains two further problem variations, concrete examples, and some background results and proofs.

\chapter{Preliminaries}\label{chpt:preliminaries}
This preliminary chapter serves to record in an organized manner some facts that form the foundation for the rest of this thesis. The prime concept is the complete homogeneous symmetric polynomial. We review its definition, some basic properties and the Stirling numbers of the second kind as a special case.

\section{Complete Homogeneous Symmetric Polynomial}\label{sec:compl_hom_symm_poly}
Let us begin by stating the classic expression defining the complete homogeneous symmetric polynomial, that is 
\begin{equation}\label{eq:def_compl_homo_poly} 
h_{d}(X_0, X_1, ..., X_k) \overset{\text{def}}{=} \sum \limits_{ \substack{c_0^{} + c_1^{} + \cdots + c_k^{} = d \\ c_j^{} \geqslant 0  } }  X_{0}^{c_0^{}} \cdot X_{1}^{c_1^{}} \cdots X_{k}^{c_k^{}} \,,
\end{equation}
for all integers $d \ge 0$ and zero otherwise \citep{macdonald1998symmetric, broder1984r-stirling}. 
The polynomial is symmetric as it is invariant under permutation of the variables $X_j$, it is homogeneous because 
the sum of exponents in each term is constant, and complete as it contains all such terms.

An alternative, equivalent form is the following:
\begin{equation}\label{eq:def_compl_homo_poly_whut}
h_{d}(X_0, X_1, ..., X_k) \overset{\text{def}}{=} \sum \limits_{0 \le i_1 \leqslant i_2 \leqslant \cdots \leqslant i_d \le k}  X_{i_1} \cdot X_{i_2} \cdots X_{i_d} \,,
\end{equation}
for all integer $d \ge 0$, and zero otherwise.

We proceed with a few useful properties.
\begin{lemma}\label{lem:compl_hom_symm_poly_recur}
The complete homogeneous polynomial $h_d(X_0, ..., X_k)$ satisfies
\begin{equation}\label{eq:compl_hom_symm_poly_recur}
h_d(X_0, ..., X_j, ..., X_k) = X_j h_{d-1}(X_0, ..., X_j, ..., X_k) + h_{d}(X_0, ..., X_{j-1}, X_{j+1}, ..., X_{k}),
\end{equation}
for any $j = 0,1,...,k$.
\end{lemma}
The foregoing property \eqref{eq:compl_hom_symm_poly_recur} expresses that the symmetric polynomial, $h_d(\cdot)$, can be split in two parts: a part containing $X_j$ and another excluding $X_j$. This may seem like a trivial observation, but is in fact the basis for many recurrence relations.

Repeated application of \Cref{lem:compl_hom_symm_poly_recur} on the same variable, yields the following corollary.
\begin{corollary}\label{cor:compl_hom_symm_poly_recur_as_sum}
\begin{equation}
h_t(X_0, ..., X_j, ..., X_k) = \sum_{d=0}^{d^\prime} X_j^d h_{t-d}(X_0, ..., X_{j-1}, X_{j+1} ..., X_k), \quad d^\prime \le t.
\end{equation}
\end{corollary}

The final property that we discuss is an explicit formula for the complete homogeneous symmetric polynomial when the variable set, $\{X_0, X_1, ..., X_k\}$, contains only distinct values, i.e. $X_i \neq X_j$ if $i \neq j$.

For such a collection of distinct real numbers, Sylvester's identity provides the following explicit, closed form.
\begin{lemma}[Sylvester's Identity]\label{lem:sylvester_identity}
Let a collection $X_0, X_1, ... , X_k \in \mathbb{R}$, with $ X_i \neq X_j$  if  $i \neq j$, be given, then the identity
\begin{equation}\label{eq:sylvester_identity}
h_{t-k} (X_0, X_1, ..., X_k) = \sum_{j=0}^{k} X_j^t \prod_{\substack{i=0 \\ i \neq j }}^{k} \frac{1}{X_j - X_i}
\end{equation}
holds true.
\end{lemma}
\begin{proof}
See \citet{bhatnagar1999sylvester} and references therein.
\end{proof}

These are the ingredients that we need in the first part of this thesis, that is \Cref{chpt:pbp_in_disc_time}.

\section{Stirling Numbers of the Second Kind}
Stirling numbers of the second kind have a combinatorial interpretation as counting the number of distinct unordered $k$-partitions of an $n$-set. However, when seen through the lens of the complete homogeneous symmetric polynomial they are the special case in which the set of variables $S = \{X_0, X_1, ..., X_k\}$ is taken equal to the set $\{0,1,...,k\}$ of consecutive integers.

The $r$-Stirling numbers are a generalization that takes the set $S$ equal to $\{r, r+1, ..., k\}$ with $0 \le r \le k$. \cite{broder1984r-stirling} provides an extensive treatment which includes their combinatorial interpretation.

The $r$-Stirling numbers of the second kind are thus defined by
\begin{equation}\label{eq:r-stirling2_as_compl_homo_poly}
\begin{aligned}
\stirlingtwo{t}{r  ,\,  k} \overset{\text{def}}{=}  h_{d}(r, r+1, ..., k) 
&= \sum \limits_{\substack{c_r^{} + c_{r+1}^{} + \cdots + c_k^{} = d \\ c_j \geqslant 0}}  r^{c_r^{}} \cdot (r+1)^{c_{r+1}^{}} \cdots k^{c_k^{}}  \\
&= \sum \limits_{r \le i_i \leqslant i_2 \leqslant \cdots \leqslant i_d \le k}  i_1 \cdot i_2 \cdots i_d  \,, \\
\end{aligned} 
\end{equation}
with $d = t-k \ge 0$, and zero otherwise. The ordinary Stirling number of the second kind are recovered when $r$ is taken equal to $0$ or $1$.

The ordinary Stirling numbers of the second kind obey the following relation\footnote{We generalize this identity for $r$-Stirling numbers of the second kind at the end of \Cref{sec:other_trans_probs}.}
\begin{equation}\label{eq:power_is_sum_fallingfactorial_times_stirlingtwo}
n^t = \sum_{j=0}^{n} (n)_j^{} \stirlingtwo{t}{j},
\end{equation}
where $(n)_j$ denotes the falling factorial $n \cdot (n-1) \cdots (n-j+1)$. Furthermore, Stirling numbers of the second kind can be expressed as finite differences.

\subsection{Finite Differences}\label{sec:fin_difference}
The finite forward difference of a function $f(\cdot)$ around $\alpha$ with a positive increment $\varepsilon$ is formally defined by
\begin{equation}
\Delta_\varepsilon \left[ f(x) \right]_{x=\alpha} \overset{\text{def}}{=} f(\alpha+\varepsilon) - f(\alpha).
\end{equation}
When the increment $\varepsilon$ is equal to $1$, we shall omit the subscript. We further distinguish the $k$th order finite forward difference, that is
\begin{equation}\label{eq:def_fin_diff}
\Delta^k \left[ f(x) \right]_{x=\alpha} = \sum_{j=0}^k (-1)^{k-j} \binom{k}{j} f(\alpha + j).
\end{equation}
For classic, extensive treatments on finite differences, we refer to \cite{boole1860fin_diff}, \cite{norlund1924fin_diff} and \cite{jordan1965fin_diff}.  

The Stirling numbers of the second kind are expressible as a finite difference in the following manner:
\begin{equation}\label{eq:stirling2_as_fin_diff}
\stirlingtwo{t}{k} = \frac{1}{k!} \Delta^k \left[ x^t  \right]_{x=0}^{}.
\end{equation}

Somewhat hidden in \citet[Sec. 3]{carlitz1980one} the $r$-Stirling numbers are expressed as a finite difference. The following explicit formula holds true
\begin{equation}\label{eq:r_stirling2_as_fin_diff}
\begin{aligned}
\stirlingtwo{t}{r, \,k} &= \frac{1}{k!} \sum_{j=0}^{k} (-1)^{k-j} \binom{k}{j} (r+j)^t 	\\
						&= \frac{1}{k!} \Delta^k \left[ x^t  \right]_{x=r} 	= \frac{1}{k!} \Delta^k \left[ (r+x)^t  \right]_{x=0}.			
\end{aligned}
\end{equation}
$r$-Stirling numbers of the second kind and their representation as a finite difference are utilized in \Cref{chpt:special_case}.

\chapter[Pure Birth Processes in Discrete Time]{Pure Birth Processes in Discrete Time}\label{chpt:pbp_in_disc_time}
\epigraph{\flushright{\itshape ``Everything should be made as simple as possible, but no simpler."}}{Albert Einstein}

In this chapter we derive our results in their greatest generality, which is as the probability distribution over the process states of a pure birth process in discrete time.

As discrete time increments, a pure birth process accumulates certain (pre-defined) events, variously called ``births", ``transitions",  or ``successes". The current process state, or stage,\footnote{The words ``stage" and ``state" are used synonymously throughout this thesis.} equals the number of ``births" that have occurred until that time. The current state of the process can at each discrete time-step ($t=0,1,2,...\,$) increase by one or remain unchanged, but it can not decrease. 

We denote the probability of a new ``birth" at stage $k$ by $p_k$. The probability of no birth during one unit of time equals $1-p_k$. The subscript $k$ indicates that these probabilities are allowed to depend on the state of the process, but that the transition probabilities are fixed and constant through time.

Perhaps the best known example of a discrete-time pure birth process is the number of successes in a sequence of independent and identical Bernoulli trials, commonly known as the binomial distribution. 

\paragraph{Chapter Outline.} We begin by formally defining the pure birth process in \Cref{sec:pbp_definition}. We then describe the duality between the state distribution and the distribution of each state's first hitting time in \Cref{sec:pbp_first_hit_times}. Importantly, the first hitting times have a simple representation as the convolution of geometric distributions. 

It is the duality between states and phases\footnote{The word ``phase'' is used as a synonym for a state's first hitting time.} that plays a crucial role when we derive the probability mass function and distribution function in \Cref{sec:pbp_results_excact_distri}. In \Cref{sec:pbp_results_ccdf_recur} we derive our main practical contribution, which is a simple recursion relation for the cumulative distribution function, and subsequently obtain its matrix implementation. We conclude with a simple procedure for the simulation a pure birth process in \Cref{sec:pbp_simulation}.

\section{Notation and Definitions}\label{sec:pbp_definition}
Discrete-time Markov chains are often depicted in a transition diagram. Pure birth processes have a particularly simple structure as evidenced from \Cref{fig:pbp_markov_chain}.

\begin{figure}[htb]
\centering
\begin{tikzpicture}[node distance = 2.1cm,  auto] 

\node[state] 						(s0) 	{0};
\node[state, right of=s0] 			(s1) 	{1};
\node[state, right of=s1] 			(s2) 	{2};
\node[draw = none, right of = s2] 	(s3_) 	{$\cdots$};
\node[state, right of=s3_] 			(s4) 	{$k$};
\node[state, right of=s4] 			(s5) 	{$k{+}1$};
\node[draw = none, right of = s5] 	(s6_) 	{$\cdots$}; 
\node[state, thick, right of=s6_] 			(s8) 	{$n$};

\draw (s0) 	edge[loop below] node {$q_0^{}$}  		(s0);
\draw (s1) 	edge[loop below] node {$q_1^{}$} 		(s1);
\draw (s2) 	edge[loop below] node {$q_2^{}$}		(s2);
\draw (s4)	edge[loop below] node {$q_{k}^{}$}		(s4); 
\draw (s5) 	edge[loop below] node {$q_{k{+}1}^{}$}	(s5); 
\draw (s8) 	edge[loop below, thick] node {$1$}				(s8);

\draw[->]  (s0) 	to node [midway, above] { $p_0^{}$ } 	  (s1);
\draw[->]  (s1) 	to node [midway, above] { $p_1^{}$ } 	  (s2);
\draw[->]  (s2) 	to node 				{ $p_2^{}$ } 	  (s3_);
\draw[->]  (s3_) 	to node [midway, above] { $p_{k-1}^{}$ }  (s4);
\draw[->]  (s4) 	to node  				{ $p_{k}^{}$ }    (s5);  
\draw[->]  (s5) 	to node [midway, above] { $p_{k+1}^{}$ }  (s6_); 
\draw[->]  (s6_)	to node [midway, above] { $p_{n-1}^{}$ }  (s8);

\end{tikzpicture}
\caption{Transition diagram of a discrete-time pure birth Markov process $X_t$.}
\label{fig:pbp_markov_chain}
\end{figure}
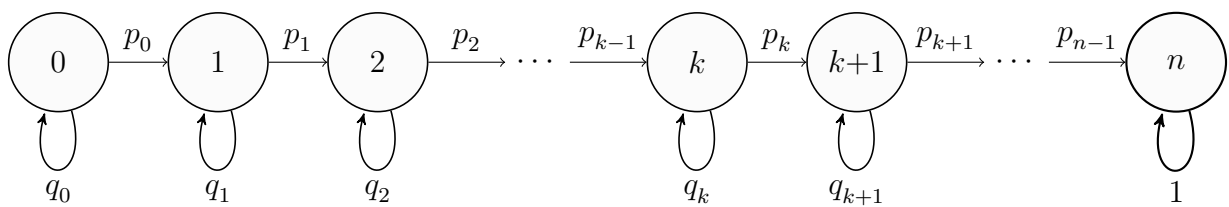

We denote the state of the process at time $t$ by the random variable $X_t$. The state space is denoted by $S$. In the case of \Cref{fig:pbp_markov_chain} $S$ equals $\{0,1,...,n\}$. We denote the collection of all successively visited states up to time $t$ by $\mathbf{X}_t \overset{\text{def}}{=} [X_0, X_1, ..., X_t]^\top$, and refer to $\mathbf{X}_t$ as a $t$-step path. We write $\mathbf{x}_t$ for a specific realization of such a path. 

We assume throughout that the process starts in state $0$ at time $t=0$, i.e. $X_0 = 0$, unless stated otherwise. Furthermore, we write $p_k$ ($k \in S$) for the probability that the process $X_t$ transitions from state $k$ to state $k+1$, and $q_k = 1 - p_k$ for the probability that the process remains in state $k$. 

To avoid trivial and degenerate cases we assume, without loss of generality and unless stated otherwise, that the transition and stop probabilities are strictly between zero and one, i.e. all $ 0 < p_k < 1$ ($k \in S \setminus \{ n \}$), with the exception of the final state which is absorbing and therefore has $q_n = 1$ and $p_n = 0$ (cf. \Cref{fig:pbp_markov_chain}). 

The vector of transition probabilities, $\mathbf{p}_{n+1} = [p_0, p_1,...,p_n]^{\top}$, thus fully defines a pure birth Markov chain.

As is customary when dealing with Markov chains, we provide its (one-step) \emph{transition matrix}, $P$. For the pure birth process in \Cref{fig:pbp_markov_chain}, the transition matrix is
\begin{equation}\label{eq:pbp_transtition_matrix}
\begin{aligned}
P &= 
\begin{bmatrix*}[c]
q_0^{}		&	p_0^{} 	  		&              	&  	 			&   	        	& 					&  						 	 	\\
	&	q_1^{}	&	p_1^{}	               &  	 			&   	        	& 					&  						 	 	\\	
	&  		&		    \small\ddots  	&     \small\ddots 	&   				&  				&  						  		\\
	&  		&    						&        q_{k-1}^{}    	&     p_{k-1}^{}	&		&   					  		\\
	&  		&    						&                  	&     q_k^{}	&	p_k^{}	&   					  		\\
	&  		&	 	 					&					&					& \small\ddots  	& 		\small\ddots	    	\\
	&  		&		      				&					&					&					&	q_{n-1}^{}	&	p_{n-1}^{}		\\
	&  		&		      		  		&					&					&					&	    		 		& 1	
\end{bmatrix*}.
\end{aligned}
\end{equation}
Formally, $P$ is defined by 
\begin{equation}\label{eq:pbp_transtition_matrix_formal}
\left[ P \right]_{i,j}^{} =
\begin{dcases} 
q_i^{}   & \mbox{if }\, j = i; 	\\ 
p_i^{}   & \mbox{if }\, j = i + 1;	\\
0    	 & \mbox{otherwise};
\end{dcases}
\end{equation}
for all $0 \leq i,j < n$; and note $[P]_{n,n}^{} = q_n = 1$. The notation $[A]_{i,j}^{}$ denotes the element in the $i$th row and $j$th column of $A$. Furthermore, $\left[ A \right]_{i,:}$ and $\left[ A \right]_{:,j}$ denote the entire $i$\textsuperscript{th} row and $j$\textsuperscript{th} column of $A$, respectively. The notation $p_{i,j}^{(t)}$ refers to the element in the $i$th row and $j$th column of $P^t$, i.e. $[P^t]_{i,j}^{}$ . These matrix elements are equal to the conditional probability $\Pr(X_{t+s} = j \,\vert\, X_s = i)$,\, $s,t \ge 0$.

Finally, from the identity $[P^{t+1}]_{i,j}^{} = [P^{t}]_{i,:}^{} [P]_{:,j}^{} $ we get following system of recurrence relations\footnote{These may be recognized as the (one-step) Chapman-Kolmogorov equations.} as an equivalent representation of the  transition matrix: 
\begin{equation}\label{eq:pbp_transitions_recur}
p_{i,j}^{(t+1)} = 
\begin{dcases} 
(1-p_0) p_{0,0}^{(t)}  		& \mbox{if }\;  i=j=0; 		\\ 
(1-p_j^{}) p_{i,j}^{(t)}    \,+\,    p_{j-1}^{}p_{i,j-1}^{(t)}  	& \mbox{if } 0 < i \le j \le n \text{ and } j < t; 		\\ 
0 							& \mbox{otherwise}; 		\\ 
\end{dcases}
\end{equation}
for $i,j \in S = \{0,1,2...,n\}$ and $t>0$; and initially $p_{i,i}^{(0)} = 1$, $i \in S$.   

In addition, we use the following notation for the Markov state probabilities (i.e. the probability that the process $X_t$ is in a certain state): 
\begin{itemize}
\item For the probability mass function, we write
\begin{equation}
f(k,t,\mathbf{p}_{n+1}) \overset{\text{def}}{=} \Pr(X_t = k \,\vert\, X_0 = 0, \mathbf{p}_{n+1}) , \quad k \in S,
\end{equation}
as well as the shorthand $\Pr(X_t = k)$.
\item For the cumulative distribution function, we write
\begin{equation}
F(k,t,\mathbf{p}_{n+1}) \overset{\text{def}}{=} \Pr(X_t \le k \,\vert\, X_0 = 0, \mathbf{p}_{n+1}) , \quad k \in S,
\end{equation}
as well as the shorthand $\Pr(X_t \le k)$. 
\item Likewise, for the complementary distribution function, we write
\begin{equation}
\overline{F}(k,t,\mathbf{p}_{n+1}) \overset{\text{def}}{=} \Pr(X_t > k \,\vert\, X_0 = 0, \mathbf{p}_{n+1}) , \quad k \in S,
\end{equation}
as well as the shorthand $\Pr(X_t > k)$. Clearly, $F(k,t,\mathbf{p}_{n+1}) + \overline{F}(k,t,\mathbf{p}_{n+1}) = 1$. For convenience we also define $\overline{F}^*(k,t,\mathbf{p}_{n+1}) \overset{\text{def}}{=} \Pr(X_t \ge k \,\vert\, X_0 = 0, \mathbf{p}_{n+1})$.
\end{itemize}
When there is no risk for misinterpretation, we shall often omit the dependence on $\mathbf{p}_{n+1}$ and $\{X_0 = 0\}$ in the notation. When we speak of the ``state distribution" or ``stage distribution", we consider the above as functions in their first argument, that is in $k$. When we wish to emphasize the argument in which the function is evaluated with respect to the other parameters being held fixed, we shall underline that variable. For example, $F(k,\underline{t},\mathbf{p}_{n+1})$ denotes the function $F$ viewed as a function of $t$ while holding the other parameters, i.e. $k$ and $\mathbf{p}_{n+1}$, fixed.

\section{First Hitting Times}\label{sec:pbp_first_hit_times}
There is a close relationship between the state distribution of the pure birth process $\mathbf{X}$ and the distribution of each state's first hitting time. Importantly, the first hitting time of any state $k$ can be expressed as the sum of \emph{independent} geometrically distributed variables. Letting $T_k$ denote the first hitting time of state $k$, we have
\begin{equation}\label{eq:pbp_first_hit_time_sum_geoms} 
T_k = G_{p_0} + G_{p_1} + \cdots + G_{p_{k-1}}, 
\end{equation} 
where each summand is geometrically distributed, that is $\Pr(G_p = t) = {p(1-p)^{t-1}}$. The time the process spends in a state, $G_p - 1$, is often called waiting time (or sojourn time). The expectation and variance of $T_k$ follow easily from the independence of the random variables. For reference, we provide them here: 
\begin{align}\label{eq:pbp_first_hit_time_exp_var}
\E[T_k] = \sum_{i=0}^{k-1} \frac{1}{p_i}\,,   \qquad \V[T_k] = \sum_{i=0}^{k-1} \frac{1-p_i}{p_i^2}. 
\end{align}

A crucial observation is that the event $\{ T_k = t \}$ is equivalent to the joint event \newline $\{X_{t-1} = k-1\} \cap \{X_t = k\} $, for any state $k = 1,...,n$. In simple terms, this equivalence expresses that in order to hit state $k$ for the first time at time $t$, the process $\mathbf{X}$ must be in state $k-1$ at time $t-1$ and transition to state $k$ at time $t$. Hence, 
\begin{equation}\label{eq:pbp_pmf_phase_in_terms_of_stage}
\begin{aligned}
\Pr(T_k = t) &= \Pr( \{ X_{t-1} = k-1\} \cap \{X_t = k \} ) \\
			 &= \Pr(X_{t-1} = k-1) 	\cdot  \Pr(X_t = k \, \vert \,  X_{t-1} = k-1 ) \\
			 &= \Pr(X_{t-1} = k-1)  \cdot  p_{k-1}^{},
\end{aligned}
\end{equation}
for all $k = 1,2,...,n$ and $t \in \mathbb{N}_{\geqslant 1}^{}$. It then follows that the composite events $\{T_k > t\}$ and $\{X_t \le k - 1\}$ are equivalent as well. Therefore, the distribution functions of $T_k$ and $X_t$ are related by
\begin{equation}\label{eq:pbp_cdf_phase_equiv_stage}
\begin{aligned}
\Pr(T_k > t)  = \Pr(X_t \le k - 1),  \quad \text{and} \quad \Pr(T_k \le t) = \Pr(X_t > k - 1). 
\end{aligned}
\end{equation}
It follows that there are three distinct ways to express the complementary first hitting time distribution:\footnote{Note: the complementary first hitting time distribution function is often called \emph{survival function}.}
\begin{equation}\label{eq:pbp_ccdf_first_hit_time_as_finite_sum_of_stage_pmf}
\begin{aligned}
\Pr(T_k > t) 			&\overset{(1)}{=} 	\sum_{d=t+1}^\infty \Pr(T_k = d)	& \text{(by definition)}\\ 
						&\overset{(2)}{=} 	p_{k-1}^{} \sum_{d=t}^\infty \Pr(X_{d} = k-1) & \text{(by Eq.  \eqref{eq:pbp_pmf_phase_in_terms_of_stage})}	\\ 
						&\overset{(3)}{=}  \Pr(X_t \le k - 1) 	=  \sum_{j=0}^{k-1} \Pr(X_{t} = j).  &  \text{(by Eq. \eqref{eq:pbp_cdf_phase_equiv_stage})}
\end{aligned}
\end{equation}

\section{Results}\label{sec:pbp_results}
Having set the stage in the \nameref{chpt:preliminaries} chapter and the foregoing sections, we are ready to present our results. The results concern the exact probability mass function and cumulative distribution function of the pure birth process state, $X_t$.\footnote{The keen reader will be able to infer the results for the first hitting time distributions without much effort using relations \eqref{eq:pbp_pmf_phase_in_terms_of_stage}, \eqref{eq:pbp_cdf_phase_equiv_stage}, and \eqref{eq:pbp_ccdf_first_hit_time_as_finite_sum_of_stage_pmf}.} 

\subsection{Exact State Distribution}\label{sec:pbp_results_excact_distri}
In this section we derive the exact state distribution of the  generic pure birth process in discrete time. We do so using a weighted path counting argument that results in the \hyperref[sec:compl_hom_symm_poly]{complete homogeneous} symmetric polynomial.

\paragraph{Paths and their Probability.} We start with a straightforward description of the probability of a path that the process $X_t$ may take. Let the process start in state $0$ at time $t=0$, that is $X_0 = 0$. Without loss of generality, assume that the process has transitioned to state $k$ in $t$ discrete time steps, i.e. $X_t = k$.

Any path $\widetilde{\mathbf{X}}_t$ that ends in state $k$, must include the transitions $0 \to 1, 1 \to 2, ..., (k-1) \to k $ as a subpath. These transitions happen in $k$ distinct time steps, leaving $t - k$ time units to be distributed as stops over the states $\{0,1,...,k\}$.

The transitions on this $(0 \to k)$-path contribute $p_0 \cdot p_1 \cdots p_{k-1}$ probability weight to a path ending in state $k$. We denote a choice of $t-k$ stops on this path by $c_0, c_1, ..., c_k$, under the conditions $c_0 + c_1 + \cdots + c_k = t - k$ and $c_j \ge 0$ ($j=0,1,...,k$).

Because each step that the process takes is independent of the previous steps (Markov property), the probability that the process takes this specific path, denoted $\widetilde{\mathbf{x}}_t^*$, equals
\begin{equation}
\Pr(\widetilde{\mathbf{X}}_{t} = \widetilde{\mathbf{x}}^*_{t}) = p_0 \cdot p_1 \cdots p_{k-1} \cdot q_0^{c_0^{*}} \cdot q_1^{c_1^{*}} \cdots q_k^{c_k^{*}}.
\end{equation}
Now, the event $\{X_t = k\}$, short for $\{X_0 = 0\} \cap \{X_t = k\}$, is the union of all paths starting in state $0$ at time $0$ and ending in state $k$ at time $t$. These paths are mutually exclusive, therefore we have
\begin{equation}\label{eq:pbp_deriv_pmf}
\begin{aligned}
\Pr(X_t = k)	&= \sum \limits_{\text{all $(0 \to k)$-paths } \widetilde{\mathbf{x}}_t} \Pr( \widetilde{\mathbf{X}}_{t} = \widetilde{\mathbf{x}}_{t}) \\
				&= \sum \limits_{ \substack{c_0^{} + c_1^{} + \cdots + c_k^{} = t-k \\ c_j^{} \geqslant 0  } } p_0 \cdot p_1 \cdots p_{k-1} \cdot q_0^{c_0^{}} \cdot q_1^{c_1^{}} \cdots q_k^{c_k^{}} \\ 
				&= \prod_{j=0}^{k-1} p_j \sum \limits_{ \substack{c_0^{} + c_1^{} + \cdots + c_k^{} = t-k \\ c_j^{} \geqslant 0  } } q_0^{c_0^{}} \cdot q_1^{c_1^{}} \cdots q_k^{c_k^{}}.
\end{aligned}
\end{equation}
Recognizing the sum in the bottom equation of \eqref{eq:pbp_deriv_pmf} as the complete homogeneous symmetric polynomial $h_{t-k}(q_0, q_1, ..., q_k)$ (cf. Eq. \eqref{eq:def_compl_homo_poly} in the \nameref{chpt:preliminaries}), we have proven the following theorem.

\begin{samepage}
\begin{theorem}[General PBP: Probability Mass Function]\label{thm:pbp_general_pmf}
Let a discrete-time pure birth process be induced by a vector of transition probabilities, $\mathbf{p}_{n+1}$, as defined in \Cref{sec:pbp_definition}; then the \emph{probability mass function} of the process state, $X_t$, is given by
\begin{equation}
f(\underline{k}, t, \mathbf{p}_{n+1}) = h_{t-k}(q_0^{}, q_1^{},..., q_k^{}) \prod_{i=0}^{k-1} p_i^{},
\end{equation}
for all $k \in S = \{0,1,...,n \}$, $t \in \mathbb{N}_{\geqslant 0}$, with $h_{t-k}(\cdot)$ the complete homogeneous symmetric polynomial.
\end{theorem}
\end{samepage}

The distribution function is by definition equal to $\Pr(X_t \le k) = \sum_{j=0}^k f(j, t, \mathbf{p}_{n+1})$. However, as we see in the next theorem, the \emph{complementary} distribution function, i.e. $\Pr(X_t > k)$, can be combined more satisfactorily with the classic theory of Markov chains with an absorbing state.\footnote{For an introduction to this topic, we refer to \cite{kemeny1976markov}.}

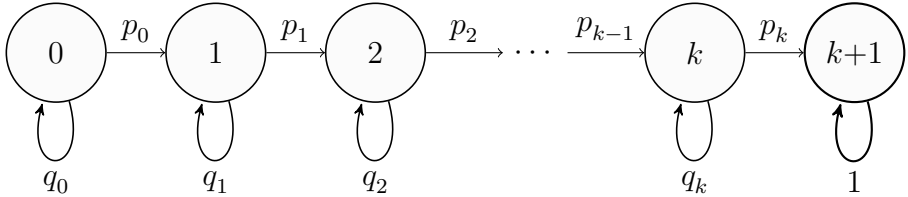
\begin{figure}[htb]
\centering
\begin{tikzpicture}[node distance = 2.1cm,  auto] 

\node[state] 						(s0) 	{0};
\node[state, right of=s0] 			(s1) 	{1};
\node[state, right of=s1] 			(s2) 	{2};
\node[draw = none, right of = s2] 	(s3_) 	{$\cdots$};
\node[state, right of=s3_] 			(s4) 	{$k$};
\node[state,thick, right of=s4] 			(s5) 	{$k{+}1$};

\draw (s0) 	edge[loop below] node {$q_0^{}$}  		(s0);
\draw (s1) 	edge[loop below] node {$q_1^{}$} 		(s1);
\draw (s2) 	edge[loop below] node {$q_2^{}$}		(s2);
\draw (s4)	edge[loop below] node {$q_{k}^{}$}		(s4); 
\draw (s5) 	edge[loop below, thick] node {$1$}				(s5); 

\draw[->]  (s0) 	to node [midway, above] { $p_0^{}$ } 	  	(s1);
\draw[->]  (s1) 	to node [midway, above] { $p_1^{}$ } 	  	(s2);
\draw[->]  (s2) 	to node 				{ $p_2^{}$ } 	  	(s3_);
\draw[->]  (s3_) 	to node [midway, above] { $p_{k-1}^{}$ }  	(s4);
\draw[->]  (s4) 	to node [midway, above] { $p_{k}^{}$ }  	(s5);

\end{tikzpicture}
\caption{Transition diagram of a truncated discrete-time Markov process with an absorbing state at $k+1$ instead of $n$.}
\label{fig:pbp_markov_chain_truncated}
\end{figure}

\begin{theorem}[General PBP: Complementary Distribution Function]\label{thm:pbp_general_ccdf}
Let a discrete-time pure birth process be induced by a vector of transition probabilities, $\mathbf{p}_{n+1}$, as defined in \Cref{sec:pbp_definition}; then the \emph{complementary distribution function} of the process state, $X_t$, is given by
\begin{equation}
\overline{F}(\underline{k}, t, \mathbf{p}_{n+1}) = h_{t-k-1}(q_0, q_1,..., q_k, 1) \prod_{i=0}^k p_i,
\end{equation}
for all $k \in S = \{0,1,...,n\}$ and $t \in \mathbb{N}_{\geqslant 0}$, with $h_{t-k-1}(\cdot)$ the complete homogeneous symmetric polynomial.
\end{theorem}
\begin{proof}
Consider the truncated Markov process of \Cref{fig:pbp_markov_chain_truncated} and denote it by $\widetilde{X}_t$. Note that
\begin{equation}
\Pr(X_t = j) = \Pr(\widetilde{X}_t = j)
\end{equation}
holds for all $j = 0,1,...,k$ and $t \ge 0$. Furthermore, we have (by definition)
\begin{equation}
\sum_{j=0}^n \Pr(X_t = j) = \sum_{j=0}^k \Pr(X_t = j) + \sum_{j=k+1}^n \Pr(X_t = j) = 1,
\end{equation}
and 
\begin{equation}
\Pr(\widetilde{X}_t = k+1) + \sum_{j=0}^k\Pr(\widetilde{X}_t = j) = 1.
\end{equation}
Therefore, 
\begin{equation}
\Pr(\widetilde{X}_t = k+1) = \sum_{j=k+1}^n \Pr(X_t = j) \overset{\text{def}}{=} \overline{F}(k, t, \mathbf{p}_{n+1}),
\end{equation}
as desired.
\end{proof}

With \Cref{thm:pbp_general_pmf} we gave a new natural proof of \cite{chen2016conv_negbinom}'s Theorem 2 -- which they derived using generating functions in the context of the convolution of geometric distributions. That the resulting expression comprises the complete homogeneous symmetric polynomial seems to have gone unnoticed by these authors. 

We were unable to find \Cref{thm:pbp_general_ccdf} in the literature. We wish to note that, even tough \Cref{thm:pbp_general_ccdf} may seem unsurprising from the perspective of the embedding Markov chain, it yields new expressions when applied to the classical occupancy distribution (\Cref{chpt:special_case}) and the binomial distribution (Appendix \ref{apx:binom_distribution}).

\subsubsection{Distinct Transition Probabilities}\label{sec:pbp_results_excact_distri_distinct}
Now we refine Theorems \ref{thm:pbp_general_pmf} and \ref{thm:pbp_general_ccdf} for the pure birth process having distinct transition probabilities, that is $p_i \neq p_j$ if $i \neq j$. We derive explicit formul{\ae} in this case by applying Sylvester's identity (\Cref{lem:sylvester_identity}) to the preceding theorems.

\begin{proposition}[PBP distinct transitions: Probability Mass Function]\label{prop:pbp_pmf_distinct}
Let a discrete-time pure birth process be induced by the transition probability vector, $\mathbf{p}_{n+1}$ ($\mathbf{q}_{n+1} = \mathbf{1} - \mathbf{p}_{n+1}$), containing only \emph{distinct} elements in $(0,1)$; then the \emph{probability mass function} of the process state, $X_t$, equals
\begin{equation}
f(\underline{k}, t,\mathbf{p}_{n+1}) = \left( \prod_{j=0}^{k-1} p_j^{} \right) \sum_{j=0}^{k} q_j^{t-k} \prod_{\substack{i=0 \\ i \neq j }}^{k} \frac{1}{q_j^{} - q_i^{}} \,.
\end{equation}
\end{proposition}
\begin{proof} 
Applying \hyperref[lem:sylvester_identity]{Sylvester's identity} to \Cref{thm:pbp_general_pmf}, the result is immediate.
\end{proof}
\Cref{prop:pbp_pmf_distinct} is equivalent to \citet[Thm. 1]{sen1999convolution_geoms}, but the derivation is shortened to one line, and motivated in a natural way from the underlying pure birth Markov chain. 

The next proposition extends this result to the cumulative distribution function.

\begin{proposition}[PBP distinct transitions: Cumulative Distribution Function]\label{prop:pbp_cdf_distinct}
Let a discrete-time pure birth process be induced by the transition probability vector, $\mathbf{p}_{n+1}$ ($\mathbf{q}_{n+1} = \mathbf{1} - \mathbf{p}_{n+1}$), containing only \emph{distinct} elements in $(0,1)$; then the \emph{cumulative distribution function} of the process state, $X_t$, equals
\begin{equation}
F(\underline{k}, t,\mathbf{p}_{n+1}) = \left( \prod_{j=0}^{k} p_j^{} \right) \sum_{j=0}^{k} \frac{q_j^{t-k}}{1-q_j} \prod_{\substack{i=0 \\ i \neq j }}^{k} \frac{1}{q_j^{} - q_i^{}} \,.
\end{equation}
\end{proposition}
\begin{proof}
We proceed from the state-phase duality as summarized in \eqref{eq:pbp_ccdf_first_hit_time_as_finite_sum_of_stage_pmf}. Hence, we have
\begin{equation}
\begin{aligned}
F(\underline{k}, t, \mathbf{p}_{n+1})  	&= \Pr(X_t \le k \,\vert\, \mathbf{p}_{n+1}) = \Pr(T_{k+1} > t \,\vert\, \mathbf{p}_{n+1}) & \text{(by Eq. \eqref{eq:pbp_ccdf_first_hit_time_as_finite_sum_of_stage_pmf})}	\\
									  	&= \sum_{d = t+1}^\infty \Pr(T_{k+1} = d \,\vert\, \mathbf{p}_{n+1})		\\
									  	&= p_k \sum_{d =t}^\infty \Pr(X_{d} = k \,\vert\, \mathbf{p}_{n+1})	& \text{(by Eq. \eqref{eq:pbp_pmf_phase_in_terms_of_stage})}	\\
									  &= p_k \sum_{d =t}^\infty \left( \prod_{j=0}^{k-1} p_j \right) \sum_{j=0}^{k} q_j^{t-k} \prod_{\substack{i=0 \\ i \neq j }}^{k} \frac{1}{q_j - q_i}	& \text{(by \Cref{prop:pbp_pmf_distinct})} \\
									  &= \left( \prod_{j=0}^{k} p_j \right) \sum_{j=0}^{k} \left( \sum_{d=t}^\infty q_j^{d-k} \right) \prod_{\substack{i=0 \\ i \neq j }}^{k} \frac{1}{q_j - q_i} & \text{(rearrange terms)}\\
									  &=  \left( \prod_{i=0}^{k} p_j \right) \sum_{j=0}^{k} \left( \frac{q_j^{t-k}}{1-q_j} \right) \prod_{\substack{i=0 \\ i \neq j }}^{k} \frac{1}{q_j - q_i}, & \text{(geometric series)}
\end{aligned}
\end{equation}
as desired.
\end{proof}
Though it is pleasant to have an explicit formula for the distribution function as in this case, the expression involving the complete homogeneous symmetric polynomial (\Cref{thm:pbp_general_ccdf}) is perhaps more practical as it gives rise to a simple linear system of recurrence relations, as we see in the next section.

\subsection{Recurrence Relations for the Distribution Function}\label{sec:pbp_results_ccdf_recur}
A completely novel result that we present in this thesis is that the cumulative distribution function of any pure birth process obeys a \emph{sparse bidiagonal system of linear recurrence relations} which is almost identical to the one that the probability mass function obeys (cf. recurrence relations in \eqref{eq:pbp_transitions_recur}). This sparse linear system makes exact and efficient computation feasible in large systems.

We present the result for the complementary distribution function here first because of its sparsity and its direct relation to \Cref{thm:pbp_general_ccdf}. 

\newpage
\begin{proposition}[Complementary CDF: Recurrence Relation]\label{prop:pbp_ccdf_recur}
The complementary distribution function of the process state, $X_t$, of any pure birth process induced by the vector of transition probabilities, $\mathbf{p}_{n+1}$, with $0 < p_i < 1$ ($i=0,1,...,n-1$), obeys the recurrence relations 
\begin{equation}
\overline{F}(k,t+1,\mathbf{p}_{n+1}) = 
\begin{dcases}
p_{k}^{} \overline{F}(k-1,t,\mathbf{p}_{n+1}) +  (1 - p_{k}^{}) \overline{F}(k,t,\mathbf{p}_{n+1}) 	& \textnormal{if } \;  k \le t;  \\
0 																									& \textnormal{if } \;  k > t;  
\end{dcases}
\end{equation}
for $1 \le k \le n$ and $t > 0$.
\end{proposition}
\begin{proof}
Starting with the latter condition ($k > t$): the complementary distribution function must be zero as it is impossible for the process $X_t$ to reach any state greater than $t$.

For the first condition ($k \le t$), we obtain two different expressions: one for the p.m.f., $f(k,t,\mathbf{p}_{n+1})$, in terms of the complementary c.d.f.\,; the other for the complementary c.d.f. when time is incremented by one unit. Upon combining these expressions the recurrence relation is established.
\begin{itemize}
\item By definition we have (in the state space):
\begin{equation}\label{eq:proof:pbp_pmf_def_as_ccdf_difference}
f(k,t,\mathbf{p}_{n+1}) = \overline{F}(k-1,t,\mathbf{p}_{n+1}) - \overline{F}(k,t,\mathbf{p}_{n+1}).
\end{equation}
\item  On the other hand, we have (in the time domain):
\begin{equation}\label{eq:proof:pbp_ccdf_time_forward}
\begin{aligned}
\overline{F}(k,t+1,\mathbf{p}_{n+1}) = \overline{F}(k,t,\mathbf{p}_{n+1}) + p_{k}^{} f(k,t,\mathbf{p}_{n+1}), 
\end{aligned}
\end{equation}
as a consequence of relations \eqref{eq:pbp_pmf_phase_in_terms_of_stage} and \eqref{eq:pbp_cdf_phase_equiv_stage}.
\item Hence, on substituting equation \eqref{eq:proof:pbp_pmf_def_as_ccdf_difference} into \eqref{eq:proof:pbp_ccdf_time_forward}, we obtain
\begin{equation}
\overline{F}(k,t+1,\mathbf{p}_{n+1}) =  p_{k}^{} \overline{F}(k-1,t,\mathbf{p}_{n+1}) + (1 - p_{k}^{})\overline{F}(k,t,\mathbf{p}_{n+1}),
\end{equation}
as desired.
\end{itemize}
\end{proof}
Note that, \textit{mutatis mutandis} the same relations hold for $\overline{F}^*(k,t,n) = {\Pr(X_t \ge k \,\vert\, X_0 = 0)}$, as a consequence of $\overline{F}^*(k,t,n) = \overline{F}(k-1,t,n)$ for $1 \le k \le n$. 

Furthermore, the regular distribution function satisfies an almost identical system of recurrence relations, which we state next. 
\begin{proposition}[CDF: Recurrence Relation]\label{prop:pbp_cdf_recur}
The distribution function of the process state, $X_t$, of any pure birth process induced by the vector of transition probabilities, $\mathbf{p}_{n+1}$, with $0 < p_i < 1$ ($i=0,1,...,n-1$), obeys the recurrence relations
\begin{equation}
F(k,t+1,\mathbf{p}_{n+1}) = 
\begin{dcases}
p_{k}^{} F(k-1,t,\mathbf{p}_{n+1}) + (1 - p_{k}^{}) F(k,t,\mathbf{p}_{n+1}) & \textnormal{if } \;  k < t;  \\
1 																			& \textnormal{if }  \; k \ge t;  
\end{dcases}
\end{equation}
for $1 \le k \le n$ and $t > 0$.
\end{proposition}
\begin{proof}
Virtually identical to the proof of \Cref{prop:pbp_ccdf_recur} (nonetheless, provided in Appendix \ref{apx:proofs}).
\end{proof}
We prefer the system of recurrence relations for the complementary c.d.f. (Prop. \ref{prop:pbp_ccdf_recur}) to the one for the c.d.f. (Prop. \ref{prop:pbp_cdf_recur}), because the former is \emph{sparse} owing to its second condition ($k > t$) being equal to zero. 

\paragraph{Matrix implementation.} The recurrence relations bear a striking resemblance to those that the probability mass function obeys (cf. Eq. \eqref{eq:pbp_transitions_recur}). It is only natural then to translate \Cref{prop:pbp_ccdf_recur} to an appropriate matrix form.

A possible direct approach would be to take \Cref{thm:pbp_general_ccdf} as a starting point, and to consider for each state $k$ the truncated chain of \Cref{fig:pbp_markov_chain_truncated} as a way of generating the corresponding value of the complementary distribution function. In that case, one would end up with as many matrices as evaluations of $\overline{F}(\cdot,t,\mathbf{p}_{n+1})$, which is undesirable.

A more economic alternative is to find an implementation of \Cref{thm:pbp_general_ccdf} that accrues all different values of $\overline{F}(\cdot,t,\mathbf{p}_{n+1})$ in a single matrix. This can be achieved \emph{matrix-algebraically} by ``placing the absorbing state up front." The matrix that achieves this, is the following:
\begin{equation}\label{eq:pbp_ccdf_recur_matrixform}
C = \begin{bmatrix}
1 	& 	p_0^{} 		&   				&   			&   			&	    				&					&	  	\\
  	& 	q_0^{}		& 	p_1^{} 			&				& 				&	    				&					&	  	\\
	&				& 	\small\ddots 	& \small\ddots 	&				&						& 					&	  	\\
	&				&					&	q_{k-1}^{}	& 	p_k^{}		&						&	 	    		&     	\\ 
	&				&					&				&   q_k^{}		& 	p_{k+1}^{} 			&	 	    		&     	\\ 		
	&				&					&				&				& 	\small\ddots 		& 	\small\ddots	&  	 	\\ 	
	&				&					&				&				& 		 				&	q_{n-2}^{}				&  p_{n-1}^{}	\\ 
	&				&					&				&				& 		 				&					&  q_{n-1}^{}	\\ 	
\end{bmatrix}.
\end{equation}
Let us make some observations about this matrix. The matrix $C$ is column-stochastic, where $P$ is row-stochastic (row-sums equal one). Mnemonically, the matrix $C$ can be thought of as ``transition matrix $P$ with its main diagonal cyclically shifted one step (to the right)," that is: $\diag(q_0,..., q_{n-1}^{} ,1) \to \diag(1, q_0^{},...,q_{n-1}^{}) $. We record $C$ formally in the following proposition.
\begin{proposition}[Complementary CDF: Matrix form]\label{prop:pbp_ccdf_recur_matrixform}
The complementary distribution function, $\overline{F}^*(j,t,n) = \Pr(X_t \ge j \,\vert\, X_0 = 0, \mathbf{p}_{n+1})$, is given by the $j$\textsuperscript{th} element in the top row of the matrix $C^t$. With $C$ defined by
\begin{equation}
\left[C \right]_{i,j}^{} = \begin{cases}
	1 					& \textnormal{if } 	j=i=0; 		\\
	1-p_j^{} 		   	& \textnormal{if }  j=i\neq 0; 	\\
	p_j^{}	   			& \textnormal{if } 	j=i+1; 		\\
	0 					& \textnormal{otherwise};
\end{cases}
\end{equation}
$0 \le {i, j} \le n$.
\end{proposition}
\begin{proof} This is a direct matrix translation of \Cref{prop:pbp_ccdf_recur} and the remark immediately following its proof. To see that the linear recurrence relation defined by matrix $C$ coincides with the recurrence relation of \Cref{prop:pbp_ccdf_recur}, simply write out $[C^{t+1}]_{0,k}^{} = [C^{t}]_{0,:}^{} [C]_{:,k}^{}$. 
\end{proof}
If one is only interested in the complementary distribution of the first $k$ states, then $C$ may be constrained to its top-left $(k+1) \times (k+1)$ submatrix. Finally, note that exponentiation of $C$ has the exact same computational complexity as exponentiation of $P$ by a power $t$.

\newpage
\subsection{Simulation}\label{sec:pbp_simulation}
We conclude with a brief exposition on the simulation of a generic pure birth process. 

Though naive sequential approaches are apparent, there exists a more efficient manner based on the representation of first hitting times as the sum of independent geometric random variables (cf. Eq. \eqref{eq:pbp_first_hit_time_sum_geoms}). 
 
In exploiting this latter representation, we are greatly aided by the simple way in which a geometric random variable can be simulated. Let $U$ denote a uniform-randomly distributed random variable on the interval $(0,1)$, then
\begin{equation}\label{eq:generate_geom_rv}
G_{p^{}} \gets \left\lceil \dfrac{\log(U)}{\log(1-p^{})} \right\rceil
\end{equation}
is geometrically distributed with success probability $p$ \citep{devroye1986rvgeneration}. Here $\lceil x \rceil$ denotes the ceiling operation, which rounds the numeric argument $x$ up to the smallest integer greater than or equal to $x$.

To simulate $X_t$, the idea is to successively generate new waiting times (including the transition), until the cumulative sum of simulated waiting times exceeds the given (``allotted") time $t$. Now, say that this exceedance happens while the process is in state $k$, then the realization of the simulated process will be $x_t = k$. We make this precise in Algorithm \ref{algo:sim_pbp_stage_process}.

\begin{algorithm}[H]\label{algo:sim_pbp_stage_process}
\linespread{1.2}\selectfont
\SetAlgoLined
\SetKwInOut{Input}{input}
\SetKwInOut{Output}{output}
\Input{$\mathbf{p}_{n+1}, t > 0$}
\Output{state $i \in \{0,1,2...,n\}$}
 $i \gets 0$\;
 \While{$t > 0$ \textnormal{and} $i < n$}{
   	$U \gets \text{Unif}(0,1)$\;
  	$G \gets \ceil \left( \log(U) / \log(1-\mathbf{p}_{n+1}(i)) \right)$\;
   	$t \gets t - G$\;
   	\If{$t \ge 0$}{
      	$i \gets i + 1$\;
    }  	
 }
 \caption{Simulate Pure Birth Process $X_t$.}
\end{algorithm}

Using this approach one needs to generate only $\min(n, t)$ standard uniformly distributed random variables in the worst case. 

\clearpage
\section{Chapter Summary}
In \Cref{chpt:pbp_in_disc_time} we studied the probability distribution of the general discrete-time pure birth process on its state space. We defined the problem formally as a Markov chain in \Cref{sec:pbp_definition}. The crucial duality with first hitting times was discussed in \Cref{sec:pbp_first_hit_times}. Importantly, each state's first hitting time can be represented as a convolution of geometric distributions. 

This convolution of geometric distributions was studied by \cite{sen1999convolution_geoms}, who in that context derived the exact probability mass function for the case having distinct success (transition) probabilities. This result was generalized to generic success probabilities by \citet[Sec. 3]{chen2016conv_negbinom}, thereby removing the distinctness condition.

Independently and using new methods, we derived the exact probability distribution of the generic discrete-time pure birth process in \Cref{sec:pbp_results_excact_distri}. Our derivation is based on the embedding Markov chain of \Cref{fig:pbp_markov_chain} combined with a weighted path counting argument that results in the complete homogeneous symmetric polynomial. 

We then extended \cite{chen2016conv_negbinom}'s result to the cumulative distribution function for both generic and distinct transition probabilities in \Cref{thm:pbp_general_pmf}, \Cref{thm:pbp_general_ccdf}, and \Cref{prop:pbp_cdf_distinct}, respectively. For the probability mass function, we provided a one-line proof in \Cref{prop:pbp_pmf_distinct} by direct application of \href{lem:sylvester_identity}{Sylvester's identity}.

Finally, and practically perhaps most satisfying, the simple structure of the pure birth Markov chain lead us to the discovery and proof of a sparse bidiagonal system of recurrence relations which generate the complementary distribution function and make its precise and efficient calculation feasible in large systems: cf. \Cref{prop:pbp_ccdf_recur} and \Cref{prop:pbp_ccdf_recur_matrixform}.\footnote{Numerical experimentation in Matlab shows that, for a randomly chosen transition probability vector, $\mathbf{p}_{n+1}$, with $n = 50{,}000$ and $t = 55{,}000$, a direct implementation of \Cref{prop:pbp_ccdf_recur} takes less than a few seconds to compute.}

\paragraph{Outlook for the remainder.} In the remainder of this thesis we give a detailed discussion of some special cases that follow from specification of the transition probabilities, $\mathbf{p}_{n+1}$. In short, we discuss: classical occupancy $p_k = \frac{n-k}{n}$ (\Cref{chpt:special_case}), randomized occupancy $p_k = p \frac{n-k}{n}$ (\Cref{sec:rand_occ_model}), and the binomial distribution $p_k \equiv p$ (\nameref{chpt:apx} \ref{apx:other_variations}). 

The specific stipulation of the transition probabilities in these cases enables further simplifications and more specific conclusions, including compact expressions involving $r$-Stirling numbers of the second kind and finite differences.

Of course, all results for these special cases can be directly obtained by application of the theorems and propositions of this chapter and then simplifying the resulting expressions. 

However, \Cref{chpt:special_case} approaches the classical occupancy problem from a different angle. Instead of employing the complete homogeneous symmetric polynomial as its main tool, it starts from the eigendecomposition of the transition matrix. 

In sum, in the remainder we aim to illustrate the wider applicability of the state distribution of discrete-time pure birth processes.

\chapter[Special Case: Classical Occupancy]{The Classical Occupancy Problem}\label{chpt:special_case}
\epigraph{\flushright{\emph{``Certainly, let us learn proving, but also \emph{let us learn guessing}."}}}{G. P\'{o}lya}

Sampling is at the basis of statistics: most, if not all, of statistics can be viewed as asking if, and to which degree, statements about the sample can be generalized to the entire population.

Sampling without replacement is perhaps the most well known sampling procedure. Though the term is imprecise, we mean that each $k$-subset has equiprobability, $\sfrac{1}{\binom{n}{k}}$, of being selected in this process. We consider this ``fair" as there is no preference for any subset or any one population element.

A common implementation is a sequential one, that is: select with equiprobability any one element from the population into the sample, then remove the selected item from the population permanently before the next stage. Select again with equiprobability from the reduced population and remove the selected element, and so on until $k$ elements have been selected into the sample (and $n-k$ remain in the reduced population). This process could, for example, yield the sample $11,2,5,7,3,1$. In statistics we generally do not care about the order in which the elements were sampled, and we simply forget the order. It is inherent to this method of sampling that any order of selected elements is equally likely, and therefore we denote the sample as a set: $\{1,2,3,5,7,11\}$.

Sampling without replacement is conceptually simple and well understood. The closely related sampling with replacement is conceptually just as simple, but its probabilistic properties are quite more involved. This makes it an interesting topic for study.

We shall therefore be concerned with sampling \emph{with replacement}. This name is not precise, and we should speak about \emph{sequential} sampling with replacement. That is, at each stage all elements in the population have equiprobability of being selected. Say, we denote the population by $[n]=\{1,2,3,...,n\}$, then at any sampling stage each population element has probability $\sfrac{1}{n}$ of being selected into the sample. This procedure is \emph{sequential} by construction and naturally produces a tuple of selected elements. Contrary to sampling without replacement, this time repetitions of elements can, and are likely to, occur in the sample.

A natural property of sequential sampling with replacement is that it is equally probable to obtain the tuple $(11,2,2,3,3,1)$ as it is to obtain $(2,1,3,2,3,11)$, or any other permutation of these elements.

Forgetting the order of the elements in a tuple, but maintaining multiplicities, results in a multiset. In this example the multiset equals $\{\!\!\{1,2,2,3,3,11\}\!\!\}$. 

Though invariance of probability under permutation of the tuple's elements is a desirable property, sampling with replacement gives us a little bit more. A tuple such as $(1,1,1,1,1,2)$ is equally likely to be sampled as, indeed, \emph{any} tuple in the Cartesian product $[n]^t$ (with $t=6$ here). Comparison of these tuples makes it immediately clear that sequential random sampling with replacement does not result in a uniformly random selection from the universe of multisets.\footnote{Note: this discrepancy between uniformity over the universes of multisets vs. the universe of tuples, bears great resemblance to the difference between Maxwell-Boltzmann and Bose-Einstein statistics \citep[cf.][]{feller1968ed3}.}  

And although, we generally do not care about the order of elements in a sampled tuple, in statistics we do care about their multiplicities: \emph{$k$ distinct population elements, in general, contain more information about the population than $k$ identical elements.} 

This realization leads us directly to the question which probability distribution sequential sampling with replacement induces on the number of distinct elements sampled in this manner. 

Answering this question is our main task in this chapter. A precise answer gives us a better understanding of sampling with replacement in general and of this property in particular.  

The distribution of the number of distinct elements in a tuple resulting from random sampling with replacement has been obtained in different contexts, the oldest known occurrence being \cite{demoivre1711mensura} \citep[see also][]{todhunter1865history}. In the last century it received the name ``classical occupancy distribution" which is derived from its multi-urn model formulation (\citet{feller1968ed3, kotz_johnson1977urn_models}; seq. Sec. \ref{sec:cocc_problem_definition}, Formulation~\ref{def:classic_occ_urn_model}). 

In our opinion, the profound importance of the classical occupancy distribution is illustrated best by its statistical and combinatorial interpretations which we formulate precisely in Formulations~\ref{def:classic_occ_swr} and \ref{def:classic_occ_combinatorial} in \Cref{sec:cocc_problem_definition}.

\paragraph{Positioning.}
Sampling with replacement is a sequential process by construction. \textit{A fortiori}, this process is a discrete-time Markov process. 

To see this, let the number of distinct elements already sampled from a finite population, $[n]$, be denoted by $k$; then the probability of sampling a yet unseen element equals $p_k = \frac{n-k}{n}$, as there are $k$ already sampled elements and $n-k$ unsampled ones in the population at that stage. Hence, the $p_k$ can be viewed as transition probabilities of this process, and $q_k = \frac{k}{n}$ the corresponding stop probabilities of each state. As such, the continuation of this process \emph{depends only} on the current state $k$ (Markov property). 

Some authors noted the Markovian nature of the classical occupancy problem \citep[e.g.][]{renyi1962three_proofs,feller1968ed3}. However, to the best of our knowledge, \cite{uppuluri1971occupancy} are the only reference taking this Markovian perspective as a starting point for their investigation. We believe that this vantage point can be exploited further, as we hope to illustrate with this thesis. 

As the embedding Markov chain of the classical occupancy problem is a discrete-time pure birth chain with distinct transition probabilities, the results that we present in this chapter are all special cases of the results presented in \Cref{chpt:pbp_in_disc_time}, in particular of Propositions \ref{prop:pbp_pmf_distinct} and \ref{prop:pbp_cdf_distinct}, and Propositions \ref{prop:pbp_ccdf_recur} and \ref{prop:pbp_ccdf_recur_matrixform}. 

However, we shall derive our results here from a different point of departure, namely from the eigendecomposition of the transition matrix. This has the benefit of validating our earlier approach and providing the reader with a matrix-algebraic alternative to the previous combinatorial one.
 
\paragraph{Outline and Results.}
We begin by summarizing some equivalent problem formulations in \Cref{sec:cocc_problem_definition}. We define the embedding Markov chain formally in \Cref{sec:cocc_markov_chain}. The workhorse of this chapter is the eigendecomposition of the transition matrix of the embedding Markov chain. We derive it in \Cref{sec:eigendecomposition}. 

The eigendecomposition is put to use in \Cref{sec:cocc_results}, to first derive the classic expressions for the probability mass function and distribution function. And then, to generalize these results to their initial-state conditioned versions in \Cref{sec:other_trans_probs} and \Cref{sec:cocc_distri_func}. 

In \Cref{sec:cocc_recur_distri} the recurrence relations for complementary c.d.f. are specifically adapted for the classical occupancy distribution in \Cref{prop:ccdf_recur}, and the corresponding matrix implementation is provided in \Cref{prop:ccdf_recur_matrixform}. 

In \Cref{sec:cocc_expectation_variance} we discuss the moments of the classical occupancy distribution and provide a clean derivation of its expectation and variance based on \cite{price1946multinomial}.

We conclude with a brief treatment of \cite{harkness1969occupancy}'s slightly more general \emph{randomized} occupancy model. We answer affirmatively a question posed by its author on whether the expressions in the exponentiated transition matrix can be simplified.

\newpage
\section{Problem Definition and Formulations}\label{sec:cocc_problem_definition}
In this section we give a concise overview of some of the formulations by which the classical occupancy problem may be defined.

The classical occupancy problem traces its origins to \citet[Problems 18 and 19]{demoivre1711mensura}. It is only fitting then to provide DeMoivre's formulation here first -- not in the least place for its conceptual clarity.

\begin{definition}[DeMoivre's Die]\label{def:classic_occ_demoivre_die}
\textnormal{What is the probability of observing $k$ different faces of an $n$-faced die, when it is rolled $t$ times?}
\end{definition}
It is assumed that each face of the die has equiprobability of coming up, and further that each roll of the die is independent of the others.

A variant of this formulation, which gives the problem a central place in statistics, is the one that phrases it in terms of sampling with replacement.
\begin{definition}[Sampling with Replacement]\label{def:classic_occ_swr}
\textnormal{Given a finite population of size $n$ from which $t$ random draws with replacement are made, what is the probability of observing $k$ distinct population elements?}
\end{definition}

The problem can also be conceptualized as an urn model, in the following manner.

\begin{definition}[Single Urn Model]
\textnormal{Initially an urn contains $n$ white balls. When a ball is randomly drawn and it is white, a black ball is returned into the urn. When a black ball is randomly drawn, we return it to the urn. What is the probability of the urn containing $k$ black balls after $t$ draws?}
\end{definition}

The problem gets its current name from the following phrasing as a multi-urn model \citep{feller1968ed3,kotz_johnson1977urn_models}.
\begin{definition}[Multiple Urn Model]\label{def:classic_occ_urn_model}
\textnormal{Given $n$ distinguishable urns of unbounded capacity, and $t$ throws with distinguishable balls having equiprobability of ending up in any of the $n$ urns, what is the probability of finding $k$ urns occupied?}
\end{definition}

These formulations describe a sequential process, where at each stage a new, not yet seen, element is added to the current sample with evolving probability that depends only on the number of distinct elements already in the sample.

Finally, let $[n]$ denote the set $\{1,2,...,n\}$. And let $[n]^t$ denote the $t$-fold Cartesian product of $[n]$.\footnote{Note: $A^t \overset{\small\text{def}}{=} \{(a_1,a_2,...,a_t) \,\vert\, a_1,a_2,..., a_t \in A\}$.} Then the following is an enumerative combinatorial definition.

\begin{definition}[Enumerative Combinatorial]\label{def:classic_occ_combinatorial}
\textnormal{What is the number of tuples in $[n]^t$ that contain exactly $k$ distinct elements?}
\end{definition}
Under uniform random sampling of the tuples in $[n]^t$, and defining the favorable cases as ``tuple contains $k$ distinct elements", we obtain the classical occupancy distribution as the fraction of favorable cases to total cases.

\section{Discrete-Time Markov Chain}\label{sec:cocc_markov_chain}
The perspective taken in this thesis is that the number of distinct elements in sequential sampling with replacement defines a discrete Markov process.  Formally, let $X_t$ denote the number of distinct elements sampled from a finite population, $[n] = \{1,2, ..., n\}$, in $t \ge 0$ draws. The state space is therefore $S=\{0,1,2,...,n\}$.

To see this problem as a Markov chain we consider the following. Let us assume that the process is in state $k$ at some time $t$, i.e.  $X_t = k$ ($k \in S \setminus \{n\}$ and $t \ge 0$). The probability of sampling a yet unsampled element equals $\frac{n-k}{n}$, as there are $n-k$ unsampled elements at that stage and the sampling is uniformly random from $[n]$ at each stage. And likewise, the probability of sampling a member of $[n]$ which was already sampled equals $\frac{k}{n}$. Finally, note that in state $k=0$ this probability is $1$, and for $k=n$ this probability is $0$. 

These simple observations allow us to define a discrete-time Markov process, as the transition probabilities only depend on the current state of the sample and not on its past. 

\begin{figure}[htb]
\centering
\begin{tikzpicture}[node distance = 2.1cm,  auto] 

\node[state] 						(s1) 	{0};
\node[state, right of=s1] 			(s2) 	{1};
\node[state, right of=s2] 			(s3) 	{2};
\node[draw = none, right of = s3] 	(s4_) 	{$\cdots$};
\node[state, right of=s4_] 			(s5) 	{$k$};
\node[state, right of=s5] 			(s6) 	{$k{+}1$};
\node[draw = none, right of = s6] 	(s7_) 	{$\cdots$}; 
\node[state, thick, right of=s7_] 			(s8) 	{$n$};

\draw (s2) 	edge[loop below] node {$\frac{1}{n}$} 	(s2);
\draw (s3) 	edge[loop below] node {$\frac{2}{n}$}	(s3);
\draw (s5)	edge[loop below] node {$\frac{k}{n}$}	(s5); 
\draw (s6) 	edge[loop below] node {$\frac{k+1}{n}$}	(s6); 
\draw (s8) 	edge[loop below, thick] node {$1$}				(s8);

\draw[->]  (s1) 	to node [midway, above] { $1$ } 	  			(s2); 
\draw[->]  (s2) 	to node [midway, above] { $\frac{n-1}{n}$ } 	(s3);
\draw[->]  (s3) 	to node 				{ $\frac{n-2}{n}$ } 	(s4_);
\draw[->]  (s4_) 	to node [midway, above] { $\frac{n-k+1}{n}$ }  	(s5);
\draw[->]  (s5) 	to node  				{ $\frac{n-k}{n}$ }    	(s6);  
\draw[->]  (s6) 	to node [midway, above] { $\frac{n-k-1}{n}$ }  	(s7_); 
\draw[->]  (s7_)	to node [midway, above] { $\frac{1}{n}$ }  		(s8);

\end{tikzpicture}
\caption{Transition diagram of the classical occupancy chain over state space $S = \{0,1,2,...,n\}$.}
\label{fig:classic_occpancy_markov_chain}
\end{figure}

The Markov chain in \Cref{fig:classic_occpancy_markov_chain} has the corresponding (one-step) transition matrix:
\begin{equation}\label{eq:matrix:1-step-transition}
\begin{aligned}
P &=  \frac{1}{n} 
\begin{bmatrix*}[c]
0	&	n 	&  		&                 	&  	 				&   	        	& 					&  						& 	 	\\
	&	1 	& n-1	&                	&  	 				&   	        	& 					&  						& 	 	\\	
	&  		&  2	& n-2              	&  	 				&   		    	& 					& 						&  		\\
	&  		&		&    \small\ddots  	&     \small\ddots 	&   				&  					&  						&  		\\
	&  		&	 	& 				   	&		 k-1		&		n-k+1 		&					&  						&  		\\
	&  		&    	&					&                  	&     k     	 	&    		  n-k	&   					&  		\\
	&  		&	 	& 					&					&					& \small\ddots  	& 		\small\ddots	&    	\\
	&  		&		&      				&					&					&					&	n-1    	            & 1		\\
	&  		&		&      		  		&					&					&					&	    		 		& n	
\end{bmatrix*}.
\end{aligned}
\end{equation}
Formally, we define the transition matrix by 
\begin{equation}\label{eq:transition_probs}
\left[ P \right]_{i,j}^{} =   
\begin{dcases} 
\frac{i}{n}     & \mbox{if }\, j = i; 	\\ 
\frac{n-i}{n}   & \mbox{if }\, j = i + 1;	\\
0    	        & \mbox{otherwise};
\end{dcases}
\end{equation}
for all $0 \leq i,j \leq n$. The notation $[A]_{i,j}^{}$ denotes the element on the $i$th row and $j$th column of $A$. Furthermore, $\left[ A \right]_{i,:}$ and $\left[ A \right]_{:,j}$ denote the entire $i$\textsuperscript{th} row and $j$\textsuperscript{th} column of $A$, respectively. We let $p_{i,j}^{(t)}$ denote the element in the $i$th row and $j$th column of $P^t$, i.e. $[P^t]_{i,j}$.

The following system of recurrence relations is an equivalent representation of this Markov chain:
\begin{equation}\label{eq:pmf_recur}
p_{i,j}^{(t+1)} = 
\begin{dcases} 
0^{t+1} 						& \mbox{if } i=j=0;  		\\
\frac{j}{n}p_{i,j}^{(t)}  \,+\, \frac{n-j+1}{n}p_{i,j-1}^{(t)} 	& \mbox{if }\;  0 < i \le j \le n \text{ and } j < t; 	\\ 
0 							    & \mbox{otherwise}; 		\\ 
\end{dcases}
\end{equation}
for all $i,j \in S = \{0, 1, ..., n\}$ and $t > 0$; and initially $p_{i,i}^{(0)} = 1$, $i \in S$. 

Finally, we note that the process defined by $Y_t = n - X_t$ is a Markov process as well. $Y_t$ equals the number of unsampled elements from $[n]$ at time $t$. This process starts in $n$, i.e. $Y_0 = n$, and then transitions to $Y_1 = n-1$, et cetera. In essence it is the same chain as the one described by Eqs.  \eqref{eq:matrix:1-step-transition}, \eqref{eq:transition_probs}, and \eqref{eq:pmf_recur} with relabeled states.

\subsection{Eigendecomposition}\label{sec:eigendecomposition} 
In this section we derive the eigendecomposition of the transition matrix, $P$. This allows us to obtain powers of $P$ from its diagonalization, $P^{t} = U\Lambda^{t} U^{-1}$, where $\Lambda$ is a diagonal matrix of eigenvalues and $U$ contains $P$'s right-eigenvectors as its columns. The topmost row of this matrix yields the classical occupancy distribution. The other rows provide the initial-state conditioned probabilities of the classical occupancy distribution. 

The eigenvalues of $P$ are simply the elements on its diagonal:
\begin{equation}\label{eq:transition_mat_eigenvals} 
\lambda_j = \frac{j}{n}, \quad 0 \le j \le n.
\end{equation}
We collect these in the diagonal matrix $\Lambda \overset{\text{def}}{=} \diag(\lambda_0, \lambda_1, ..., \lambda_n) = n^{-1} \diag(0,1,...,n)$. 

Corresponding integer-valued eigenvectors were suggested by numerical experimentation. The following lemma makes this precise and provides a proof of this ``educated guess".

\begin{lemma}\label{lem:eigenvector_mat_U}
Integer-valued, right-eigenvectors of $P$ are given by the columns of the $(n+1) \times (n+1)$ matrix $U$, defined by
\begin{equation}
\left[U\right]_{{i,k}} = 
\begin{cases}
\binom{n-i}{n-k} & \textnormal{if }\, i \le k; \\
0                & \textnormal{otherwise}.
\end{cases}
\end{equation}
with $0 \le i,k \le n$.
\end{lemma}
\begin{proof}
Denote the matrix $P = [\mathbf{p}_0, \mathbf{p}_1, ..., \mathbf{p}_n]$, likewise denote $U = [\mathbf{u}_0, \mathbf{u}_1, ..., \mathbf{u}_n]$.

We must verify that $P \mathbf{u}_{k} = \lambda_k \mathbf{u}_{k}$ for all $0 \le k \le n$. We separate two cases: column $k=0$ and the remaining columns $1 \le k \le n$.

\begin{itemize}
\item \emph{The case $k = 0$}. The eigenspace belonging to eigenvalue $\lambda_0 = 0$, follows from solving the system
\begin{equation}
P \mathbf{u}_0 = \lambda_0 \mathbf{u}_0 = \mathbf{0}.
\end{equation}
We denote this system by $ [P \,\vert\, \mathbf{0}]$. Because $P$ is bidiagonal, this system is straightforwardly solved by Gaussian elimination (starting from the bottom right and working upward). The reduced row-echelon form equals
\begin{equation}
\left[\begin{array}{ll|r}
1 			& \mathbf{0}^\top 	& 0 \\
\mathbf{0} 	& O_{n \times n}    & \mathbf{0}
\end{array}\right].
\end{equation}
This system has infinitely many solutions of the form $s [1, 0, 0, ..., 0]^\top$ with $s \in \mathbb{R}$. Hence, $\mathbf{u}_0 = [1, 0, 0, ..., 0]^\top = [\binom{n}{0}, 0, 0, ..., 0]^\top$ spans the eigenspace belonging to eigenvalue $\lambda_0 = 0$.

\item \emph{The columns $1 \le k \le n$}. For the remaining eigenvalue-eigenvector pairs, we proceed with a direct demonstrative computation. Again, we must verify that 
\begin{equation}\label{eq:eigendecomp_PU_is_ULamba}
P U = U \Lambda,
\end{equation}
with $U$ as supposed and $\Lambda = n^{-1}\diag(0,1,...,n)$ the diagonal matrix of ordered eigenvalues (that is Eq. \eqref{eq:transition_mat_eigenvals}). 

We reduce both sides of \eqref{eq:eigendecomp_PU_is_ULamba} to equivalent simpler statements. Starting with the left-hand side, we may re-express it explicitly for each column $k$  ($0 \le k \le n$), as: $[P U]_{:,k} = P \mathbf{u}_k = $
\begin{equation}
\begin{aligned}
\frac{1}{n} \left[u_{0,k}^{}, u_{0,k}^{} + (n{-}1) u_{1,k}^{}, ...,\, i u_{i,k}^{} + (n{-}i) u_{i+1,k}^{},\, ... ,(k{-}1) u_{k-1,k}^{} + (n{-}k{+}1)u_{k,k}, k u_{k,k}^{}, \mathbf{0}_{n-k}^\top \right]^\top .
\end{aligned}
\end{equation}

The right-hand side of \eqref{eq:eigendecomp_PU_is_ULamba} can be written in column-partition form as:
\begin{equation}
U \Lambda = \frac{1}{n} \left[\mathbf{u}_0, \mathbf{u}_1, 2 \mathbf{u}_2,\, ...,\, k \mathbf{u}_k,\, ....,\, (n{-}1) \mathbf{u}_{n-1}, n \mathbf{u}_n \right].
\end{equation}

Now, consider any eigenvector $\mathbf{u}_k$, $1 \le k  \le n $. Comparing the elements on the left- and right-hand sides of \eqref{eq:eigendecomp_PU_is_ULamba}, the demonstrandum is reduced to proving that     
\begin{equation}
\frac{i}{n}\binom{n-i}{n-k} + \frac{(n-i)}{n} \binom{n-i-1}{n-k} = \frac{k}{n} \binom{n-i}{n-k}
\end{equation}
for $i = 0,1,...,k$ given any $k = 1,2,...,n$. Verifying this relation now becomes an algebraic exercise. Starting from the left-hand side, we deduce 
\begin{equation}
\begin{aligned}
i \frac{(n-i)!}{(n-k)!(k-i)!} + (n-i)  \frac{(n-i-1)!}{(n-k)!(k-i-1)!} &= i \frac{(n-i)!}{(n-k)!(k-i)!} +\frac{(n-i)!}{(n-k)!(k-i-1)!} \\
&= i \binom{n-i}{n-k}  +  (k-i) \binom{n-i}{n-k}  \\
&= k \binom{n-i}{n-k},
\end{aligned}
\end{equation}
as desired, completing the proof.
\end{itemize}

\end{proof}
A concrete example of $U$ is provided in Appendix \ref{apx:matrix_examples}. One recognizes $U$ as a Pascal matrix \citep{velleman1993pascal} oriented with ones on its diagonal and ones in its rightmost column. $U$'s inverse, $U^{-1}$, is identical apart from signs.
\begin{lemma}\label{lem:eigenvector_mat_U_inverse}
The integer-valued inverse of the right-eigenvector matrix $U$ of $P$ (\Cref{lem:eigenvector_mat_U}) is given by
\begin{equation}
\left[U^{-1}\right]_{{i,j}} = 
\begin{cases}
(-1)^{j-i}\binom{n-i}{n-j} & \textnormal{if }\, i \le j; \\
0                          & \textnormal{otherwise}.
\end{cases}
\end{equation}
\end{lemma}
\begin{proof}
This result follows from direct multiplication with matrix $U$ from \Cref{lem:eigenvector_mat_U}. Additionally, we refer to \citet{velleman1993pascal} or \citet{spivey2008symmetric}.
\end{proof}

Combined, the eigenvalues in Eq. \eqref{eq:transition_mat_eigenvals} and the eigenvector-matrices in Lemmas \ref{lem:eigenvector_mat_U} and \ref{lem:eigenvector_mat_U_inverse} yield $P^t = U \Lambda^t U^{-1}$, as a way to express $p_{i,j}^{(t)}$ explicitly (as we do in the \hyperref[sec:cocc_results]{Results section}).

We illustrate the usefulness of the eigendecomposition with the derivation of an elementary, but nonetheless fundamental, property of the Markov chain.
\begin{lemma}[Limit Distribution]\label{lem:limit_distri}
\begin{equation}
\underset{t \rightarrow \infty}{\lim} P^t = \left[O_{(n+1) \times n} \,\Big\vert\, \mathbf{1}_{n+1} \right].
\end{equation}
\end{lemma}
\begin{proof}
We will make use of the eigendecomposition, $P^t = U \Lambda^t U^{-1}$, and the observation that $\Lambda$ is a diagonal matrix of which all elements are smaller than $1$ except the bottom-right element. That is $\lambda_j < 1$ for all $0 \le j < n$, and $\lambda_n = 1$. Hence,
\begin{equation}
\begin{aligned}
\underset{t \rightarrow \infty}{\lim} P^t 	&= \underset{t \to \infty}{\lim} U \Lambda^t U^{-1} 	\\
											&= U \left( \underset{t \to \infty}{\lim}  \Lambda^t \right) U^{-1}		\\
											&= U \diag(0,0,...,1) U^{-1}	 		\\
											&= \left[O_{(n+1) \times n} \,\Big\vert\, \mathbf{1}_{n{+}1} \right],
\end{aligned}
\end{equation}
where $O_{(n+1) \times n}$ denotes an $(n+1) \times n$ matrix containing only zeros, and $\mathbf{1}_{n+1}$ denotes a column vector containing $n+1$ ones.
\end{proof}
In other words, \Cref{lem:limit_distri} states that the final state $n$ is absorbing, and all other states are transient.

\section{Results}\label{sec:cocc_results}
In this section we derive the exact probability mass function and distribution function of the classical occupancy problem in their initial-state conditioned forms. Furthermore, we derive an efficient system of recurrence relations for the (complementary) distribution function. 

\subsection{The Probability Mass Function} 
We shall show that the probability mass function of the classical occupancy problem can be derived from the eigendecomposition of the transition matrix, $P$.

\begin{proposition}[Probability Mass Function]\label{prop:pmf_eigen}
The probability mass function of the classical occupancy problem is given by
\begin{align}
f(k,t,n) & \,\overset{(1)}{=}\, \frac{1}{n^t} \binom{n}{k} \sum_{j=0}^k (-1)^{k-j} \binom{k}{j} j^t \label{eq:pmf_eigen1} \\ 
		 & \,\overset{(2)}{=}\, \frac{(n)_{k}^{}}{{n^t}} \stirlingtwo{t}{k} \label{eq:pmf_eigen2}
\end{align}
for $ 0 \le k \le \min(n,t)$; and $0$ otherwise.
\end{proposition}
\begin{proof}
Consider any column $k$ in the zeroth row of the eigendecomposition of $P^t$, that is
\begin{equation}\label{eq:pmf_eigen_proof1} 
\begin{aligned}
f(k,t,n) 	&=  [P^t]_{0,k}  \\
			&=  [U \Lambda^t U^{-1}]_{0,k} =  [U]_{0,:} \Lambda^t [U^{-1}]_{:,k} \\			
			&= \frac{1}{n^t} \sum_{j=0}^k (-1)^{k-j} \binom{n-j}{n-k} \binom{n}{j} j^t. & \text{(cf. Eq. \eqref{eq:apx_U_binomcoeff} in Apx. \ref{apx:matrix_examples_transmat})} 
\end{aligned}
\end{equation}
Then using
\[ 
\binom{n-j}{n-k} \binom{n}{j} = \frac{\cancel{(n-j)!}}{(n-k)!\, (k-j)!} \frac{n!}{\cancel{(n-j)!}\, j!} = \frac{n!}{(n-k)!\, k!} \; \frac{k!}{(k-j)!\, j!} = \binom{n}{k} \binom{k}{j},
\]
and substituting this expression back into the expression \eqref{eq:pmf_eigen_proof1}, we obtain
\begin{equation*}
\begin{aligned}			
f(k,t,n)    &= \frac{1}{n^t} \binom{n}{k} \sum_{j=0}^k (-1)^{k-j} \binom{k}{j} j^t, 
\end{aligned}		
\end{equation*}
as desired in \eqref{eq:pmf_eigen1}.
To prove identity \eqref{eq:pmf_eigen2}, we start from identity \eqref{eq:pmf_eigen1} and simplify as follows
\begin{equation}\label{eq:pmf_eigen_proof2}
\begin{aligned}
f(k,t,n) 	&= \frac{1}{n^t} \binom{n}{k} \sum_{j=0}^k (-1)^{k-j}  \binom{k}{j} j^t \\
			&= \frac{1}{n^t} \frac{n!}{(n-k)!} \frac{1}{k!} \sum_{j=0}^k (-1)^{k-j} \binom{k}{j} j^t \\
			&= \frac{(n)_k^{}}{n^t} \stirlingtwo{t}{k}, \qquad \text{(by Eq. \eqref{eq:stirling2_as_fin_diff})}
\end{aligned}
\end{equation}
as desired.\footnote{Note: for $t > 0$ the zeroth summand evaluates to zero, and the summation may therefore also start from $j=1$; for $t=0$ the zeroth summand evaluates to one as a result of the convention $0^0 = 1$, as is proper.}
\end{proof}

Expression \eqref{eq:pmf_eigen1} is the classic expression as obtained by \cite{demoivre1711mensura} from inclusion-exclusion arguments (see also \citet[Ch. 4.2]{feller1968ed3}). Though the result itself is not new, we have not seen its derivation from the eigendecomposition of the transition matrix before.

Finally, from Eq. \eqref{eq:pmf_eigen_proof2} we see that we can also write the expression as a finite difference:
\begin{equation}
f(k,t,n) = \frac{(n)_k}{n^t} \frac{1}{k!} \Delta^{k} \left[ x^t \right]_{x=0}^{}.
\end{equation} 

In the next section, we generalize \Cref{prop:pmf_eigen} to its initial-state conditioned version.

\subsection{The Other Multistep Transition Probabilities}\label{sec:other_trans_probs}
Exponentiation of the transition matrix $P$ directly yields the values of all multistep transition probabilities, including the initial-state conditioned ones: $[P^t]_{r,k}^{} = p^{(t)}_{r,k} = P(X_{t} = k \,\vert\, X_0 = r)$. 

In the context of DeMoivre's $n$-faced die (Formulation~\ref{def:classic_occ_demoivre_die}), these quantities can be interpreted as the probability of observing $k$ distinct faces in a total $t+s$ throws, after having observed $r$ distinct faces in the first $s$ throws. In other words, the probability $p^{(t)}_{r,k}$ precisely captures how the ``prior probability" $p^{(t+s)}_{0,k}$ needs to be updated to include the information right after throw $s$. From the matrix $P^t$ we can read-off any $p_{r,k}^{(t)}$. These are simply interpreted as $t$-step initial-state conditioned probabilities.

The eigendecomposition that we derived in \Cref{sec:eigendecomposition} directly yields explicit expressions for all these probabilities. We obtain ``closed-form" expressions in this section as well, by recognizing the $r$-Stirling numbers of the second kind in the derived expressions \citep{broder1984r-stirling, carlitz1980one}. To our knowledge these probabilities have not been obtained in this form before.

Let us note, that  these closed-form expressions can alternatively be obtained by using the techniques and results of \Cref{chpt:pbp_in_disc_time}. This follows from the observation that conditioning on the initial state results in a pure birth process again. 

\begin{proposition}[Transition Probabilities]\label{prop:transition_probs_classical_occupancy}
The initial-state conditioned probability mass function of the classical occupancy problem is given by
\begin{align}
\Pr(X_t = k \,\vert\, X_0 = r)  & \, \overset{(1)}{=} \, \frac{1}{n^t} \,  \binom{n-r}{n-k}  \sum_{j=0}^{k-r} (-1)^{k-r-j} \binom{k-r}{j} (j+r)^t \label{eq:explicit_from_eigendecomp} \\	
			 	&	\, \overset{(2)}{=} \, \frac{(n-r)_{k-r}^{}}{n^t} \stirlingtwo{t+r}{r,\;k}		\label{eq:pmf_r_stirling}		
\end{align}
for $0 \le r \le k \le \min(n,t)$; and $0$ otherwise.
\end{proposition}

We give an algebraic proof of \eqref{eq:explicit_from_eigendecomp} first, taking the eigendecomposition of the transition matrix as our starting point. We then give a combinatorial proof of \eqref{eq:pmf_r_stirling}. 

\begin{proof}[Algebraic proof of \eqref{eq:explicit_from_eigendecomp}.]
Consider the eigendecomposition of $P^t$, that is, 
\begin{equation}\label{eq:conditional_prob_eigendecomp}
\begin{aligned}
p^{(t)}_{r,k} &= [P^t]_{r,k}^{} = \left[ U \Lambda^t U^{-1} \right]_{r,k}^{} = [U]_{r,\mathbf{:}}^{} \, \Lambda^t \, [U^{-1}]_{\mathbf{:},k}^{} \\
              &= \frac{1}{n^t} \sum_{j=r}^k (-1)^{k-j} \binom{n-j}{n-k} \binom{n-r}{j-r} j^t & \text{(cf. Eq. \eqref{eq:apx_U_binomcoeff} in Apx. \ref{apx:matrix_examples_transmat})}  \\ 
			  &= \frac{1}{n^t} \sum_{j=0}^{k-r} (-1)^{k-j-r} \binom{n-j-r}{n-k} \binom{n-r}{j} (j+r)^t. & \text{(reindex)} \\
\end{aligned}
\end{equation}
Note that the product of binomial coefficients in each summand can be re-expressed as follows
\begin{equation}
\begin{aligned}
\binom{n-j-r}{n-k} \binom{n-r}{j} 	&= \frac{\cancel{(n-j-r)!}}{(k-j-r)!(n-k)!} \frac{(n-r)!}{j!\cancel{(n-j-r)!}} \\
								  	&= \frac{(n-r)!}{(n-k)!}\, \frac{1}{(k-r-j)!j!} \\
								  	&= \frac{(n-r)!}{(n-k)!}\, \frac{1}{(k-r)!} \, \frac{(k-r)!}{(k-r-j)!j!} &\text{(introduce $(k-r)!$)} \\									&= \binom{n-r}{n-k}\, \binom{k-r}{j}. 
\end{aligned}
\end{equation}
Substituting this latter expression back into equation \eqref{eq:conditional_prob_eigendecomp}, we obtain
\begin{equation}
\begin{aligned}
p^{(t)}_{r,k} = \frac{1}{n^t} \, \binom{n-r}{n-k} \sum_{j=0}^{k-r} (-1)^{k-r-j} \binom{k-r}{j} (j+r)^t. \\ 
\end{aligned}
\end{equation}
as desired in expression \eqref{eq:explicit_from_eigendecomp}.
\end{proof}
Note that \eqref{eq:explicit_from_eigendecomp} is in the form of a finite difference (cf. \Cref{sec:fin_difference}). For the second expression, \eqref{eq:pmf_r_stirling}, we give a combinatorial proof, based on the probability weight of the $t$-step $(r \to k)$-path.
\begin{proof}[Combinatorial proof of \eqref{eq:pmf_r_stirling}.]
Each $(r \rightarrow k)$-path contains the consecutive transitions $r \rightarrow (r+1)$, $(r+1) \rightarrow (r+2)$, ..., $(k-1) \rightarrow k$, which contribute probability weight
\begin{equation}\label{eq:pmf_probweight_trans}
\frac{(n-r) \cdot (n-r-1) \cdots (n-k+1)}{n^{k-r}} = \frac{(n-r)_{(k-r)}^{}}{n^{k-r}}
\end{equation}
to each such path.\footnote{Identical argument to the one used in \Cref{sec:pbp_results_excact_distri}.}

This leaves $t-k+r$ stops to be chosen on the path, under the restriction that each new stop must be further down the chain than the farthest stop already on the path. The sum of these probability weights over all possible stop configurations, equals 
\begin{equation}\label{eq:pmf_probweight_stops0}
\frac{1}{n^{d}} \sum \limits_{r \le i_1 \leqslant i_2 \leqslant ... \leqslant i_{d} \le k}  i_1 \cdot i_2 \cdots i_d = \frac{1}{n^{d}} h_{d}(r,r+1,...,k)  
\end{equation}
with $d = t-k+r$. Now, using equation \eqref{eq:r-stirling2_as_compl_homo_poly} from the \nameref{chpt:preliminaries}, we can write \eqref{eq:pmf_probweight_stops0} as
\begin{equation}\label{eq:pmf_probweight_stops}
\frac{1}{n^{t-k+r}} h_{t-k+r}(r,r+1,...,k) = \frac{1}{n^{t-k+r}} \stirlingtwo{t+r}{r,\;k}.
\end{equation}
Finally, we combine the probability weight contributed by the transitions \eqref{eq:pmf_probweight_trans} and the stops \eqref{eq:pmf_probweight_stops}, to obtain
\begin{equation*}
p^{(t)}_{r,k} = \frac{(n-r)_{k-r}^{}}{n^t} \stirlingtwo{t+r}{r,\;k},
\end{equation*}
as desired in \eqref{eq:pmf_r_stirling}.
\end{proof}

Since both expression \eqref{eq:explicit_from_eigendecomp} and \eqref{eq:pmf_r_stirling} equal $p^{(t)}_{r,k}$, these distinct proofs show that the $r$-Stirling numbers of the second kind are given by the explicit formula
\begin{equation}\label{eq:r-stirling2_explicit_from_eigendecomp}
\stirlingtwo{t+r}{r,\; k} = \frac{1}{(k-r)!} \sum_{j=0}^{k-r} (-1)^{k-r-j} \binom{k-r}{j} (j+r)^t,
\end{equation}
and furthermore as the following finite difference:\footnote{Note that this derivation is markedly different from that of \cite{carlitz1980one}.}
\begin{equation}\label{eq:r-stirling2_as_fin_diff} 
\stirlingtwo{t+r}{r,\; k} = \frac{1}{(k-r)!} \Delta^{k-r} \left[ x^t \right]_{x=r}.
\end{equation}

Concluding this section, we record a new identity for $r$-Stirling numbers of the second kind which generalizes the well known classic identity \eqref{eq:power_is_sum_fallingfactorial_times_stirlingtwo}. 

As $P$ is a row-stochastic matrix, each row of $P^t$ sums to $\frac{n^t}{n^t} = 1$. This observation leads immediately to the following sum-identity for $r$-Stirling numbers of the second kind. That is, the full row-sums of $n^t P^t$ are equal to
\begin{equation}\label{eq:r_stirling2_fallingfactorial_sum}
\sum_{j=r}^n (n)_j^{} \stirlingtwo{t+r}{r,\; j} = (n)_r^{} \, n^{t-1} 
\end{equation}
for any row $r=0, 1, 2,\, \dots, n$. The classic identity \eqref{eq:power_is_sum_fallingfactorial_times_stirlingtwo} follows immediately when we set $r$ equal to $0$.

\subsection{The Distribution Function}\label{sec:cocc_distri_func}
In this section we derive the distribution function of the classical occupancy problem. We shall do so for all initial-state conditioned distributions at once.

We give a purely matrix-algebraic proof here based on the eigendecomposition of the transition matrix, and a simplifying result for the matrix of right-eigenvectors, $U^{-1}$. Appendix \ref{apx:proofs} contains a proof based on the state-phase duality (that we discussed in \Cref{sec:pbp_first_hit_times}).\newline

We record some auxiliary results first. By definition the distribution function is the sum of probability masses over states $j=0,1,2,\,\dots, k \le n$. It is therefore convenient to capture the sum operation in matrix form.

We introduce the upper-triangular matrix $\Sigma$, defined as follows:
\begin{equation}\label{eq:def_sigma}
\left[ \Sigma\phantom{^{}} \right]_{i,j}^{} = \begin{cases}
1  & \text{if } i \le j; \\
0  & \text{otherwise};
\end{cases}
\end{equation}
for $0 \le i,j \le n$. Where it appears, the matrix can be assumed conformable to the matrix on which it acts. The matrix $\Sigma$ allows us to conveniently write $P^t \Sigma$ for the cumulative row-sums of $P^t$.

Furthermore, $\Sigma$'s inverse is given by
\begin{equation}\label{eq:def_sigma_inverse}
\left[ \Sigma^{-1} \right]_{i,j}^{} = \begin{cases}
\phantom{-}1  	& 	\text{if } j = i; 	\\
  		  -1  	& 	\text{if } j = i+1; \\
\phantom{-}0  	& 	\text{otherwise}.
\end{cases}
\end{equation}
Appendix \ref{apx:matrix_examples} contains examples of both matrices.

\begin{lemma}\label{lem:Uinv_times_Sigma} 
Let $U_n^{-1}$ denote the $(n+1) \times (n+1)$ right-eigenvector matrix of \Cref{lem:eigenvector_mat_U_inverse}, then
\begin{equation}
U_n^{-1} \Sigma = \left[ \begin{array}{c|c}
U_{n-1}^{-1} 	& \mathbf{0} \\
\hline 
\mathbf{0}^\top 	& 1
\end{array} \right]. 
\end{equation}
\end{lemma}
\begin{proof}
Consider the following cases:
\begin{enumerate}
\item For the bottom row, simply note that $\sum_{i=0}^{j} 0  = 0$ (for all $ j = 0,1,...,n-1$), and ${1 + \sum_{i=0}^{n} 0 = 1}$. 
\item For the rightmost column, apply the Binomial Theorem: $\sum_{i=0}^j (-1)^{j-i} \binom{j}{i} = (1 - 1)^{j-i} = 0$.
\item For the topleft $n \times n$ submatrix of $U_n^{-1}\Sigma$, note that each element $(i,j)$ satisfies  
\[		\sum \limits_{i=0}^j (-1)^{j-i} \binom{n-i}{n-j} =  \binom{n-i-1}{n-j-1}, \quad  0 \le j < n,      \] 
and equals zero for $j=n$. See e.g. \citet[Theorem E, p. 10]{comtet1974combinatorics} for a reference, and Appendix \ref{apx:matrix_examples} for a concrete example.
\end{enumerate}
This completes the proof.
\end{proof}
This interesting result is also quite convenient. Applying it to the eigendecomposition of $P^t$, that is $U \Lambda^t (U^{-1} \Sigma)$, it shows for one, that the distribution function in matrix form is \emph{not} more complicated than the probability mass function. And secondly, it yields explicit formul{\ae} for the conditional distribution, $\Pr(X_t \le j \,\vert\, X_0 = r)$ for any $0 \le r \le n$.

\begin{proposition}[Cumulative Occupancy Distribution]\label{prop:cdf_from_eigendecomp}
The initial-state conditioned cumulative distribution function of the classical occupancy problem is given by
\begin{equation}\label{eq:cdf_from_eigendecomp}
\begin{aligned}
\Pr(X_t \le k \,\vert\, X_0 = r) &= \frac{n-k}{n^t} \binom{n-r}{n-k} \sum_{j=0}^{k-r} (-1)^{k-r-j}  \binom{k-r}{j} \frac{(r+j)^t}{n-r-j} \\
\end{aligned}
\end{equation}
for all integers $0 \le r \le k < n$ and $t \ge 0$; $\Pr(X_0 \le n \,\vert\, X_0 = k) = 1$ for all $0 \le k \le n$; and $0$ otherwise.
\end{proposition}

\begin{proof} 
We start by considering the eigendecomposition of $P^t$. The cumulative row-sums of $P^t$ may be expressed using the sum-matrix $\Sigma$ (Eq. \eqref{eq:def_sigma}), as follows
\begin{equation}
P^t \Sigma = U \Lambda^t U^{-1} \Sigma =  U \Lambda^t V^{-1},
\end{equation}
with 
\begin{equation}
V^{-1} = \left[ \begin{array}{c|c}
				U_{n-1}^{-1} & \mathbf{0} \\
				\hline
				\mathbf{0'}  & 1
				\end{array} \right],
\end{equation}
as in \Cref{lem:Uinv_times_Sigma}.

Next, consider that $\sum_{j=0}^k [P]_{r,j}^{t}$ equals matrix element $\left[P^t \Sigma \right]_{r,k}^{}$ which is the sought probability. We make the expression explicit as follows
\begin{equation}\label{eq:conditional_cdf_eigendecomp_in_proof}
\begin{aligned}
\left[P^t \Sigma \right]_{r,k}^{} &= \left[ U \Lambda^t U^{-1} \Sigma \right]_{r,k}^{} = \left[ U \Lambda^t (U^{-1} \Sigma) \right]_{r,k}^{}  \\
							&= [ U ]_{r,\mathbf{:}}^{} \, \Lambda^t \, [V^{-1}]_{\mathbf{:},k}^{} \\
              				&= \frac{1}{n^t} \sum_{j=r}^k (-1)^{k-j} \binom{n-j-1}{n-k-1} \binom{n-r}{n-j} j^t. 
\end{aligned}
\end{equation} 
(We recommend that the reader views Eq. \eqref{eq:apx_UinvSigma_binomcoeff_explicit} in Appendix \ref{apx:matrix_examples_transmat}.)

Now consider the product of binomial coefficients in the last line of \eqref{eq:conditional_cdf_eigendecomp_in_proof}, and simplify as follows:
\begin{equation}
\begin{aligned}
\binom{n-j-1}{n-k-1} \binom{n-r}{n-j} &= \frac{\cancel{(n-j-1)!}}{(k-j)!(n-k-1)!}\; \frac{(n-r)!}{(j-r)!\cancel{(n-j-1)!}}\; \frac{1}{n-j} \\
				&= \frac{(n-r)!}{(n-k-1)!}\; \frac{1}{(k-j)!(j-r)!}\; \frac{1}{n-j}\\
				&= (n-k)\;\frac{(n-r)!}{(k-r)!(n-k)!}\; \frac{(k-r)!}{(k-j)!(j-r)!} \;\frac{1}{n-j}\\				
				&= (n-k)\; \binom{n-r}{n-k} \binom{k-r}{j-k} \frac{1}{n-j}  .
\end{aligned}
\end{equation}
Next, substitute this expression into \eqref{eq:conditional_cdf_eigendecomp_in_proof} and reindex the summation as follows
\begin{equation}\label{eq:cdf_proof_last_eq}
\begin{aligned}
\Pr(X_t \le k \vert\, X_0 = r) 	&= \frac{n-k}{n^t} \sum_{j=r}^{k} (-1)^{k-j} \binom{n-r}{n-k} \binom{k-r}{j-k} \frac{j^t}{n-j}  		\\                          								&= \frac{(n-k)}{n^t} \binom{n-r}{n-k} \sum_{j=r}^{k} (-1)^{k-j}   \binom{k-r}{j-r} \frac{j^t}{n-j}  	\\
								&= \frac{(n-k)}{n^t} \binom{n-r}{n-k} \sum_{j=0}^{k-r} (-1)^{k-j-r} \binom{k-r}{j} \frac{(r+j)^t}{n-r-j},  \quad\text{(reindex: $j \leftarrow r+j$)},
\end{aligned}
\end{equation} 
which completes the proof.
\end{proof}

Finally, note that Eq. \eqref{eq:cdf_proof_last_eq} is compactly expressible as the finite difference
\begin{equation}
\Pr(X_t \le k \,\vert\, X_0 = r) = \frac{(n-r)\cdots(n-k)}{n^t} \frac{1}{(k-r)!} \Delta^{k-r} \left[ \frac{x^t}{n-x}  \right]_{x=r}^{},
\end{equation}
with the unconditional distribution function as a special case ($r=0$):
\begin{equation}
F(k,t,n) = \frac{(n)_{k+1}^{}}{n^t} \frac{1}{k!} \Delta^{k} \left[ \frac{x^t}{n-x}  \right]_{x=0}^{}.
\end{equation}

\subsection{Recurrence Relations for the Distribution Function}\label{sec:cocc_recur_distri}
To make calculation of the distribution function feasible and thereby practically accessible, we derive an efficient algorithm for the calculation of the cumulative probabilities in this section. 

We first derive a system of linear recurrence relations for the cumulative distribution function and subsequently a bidiagonal matrix version. These recurrence relations bear a striking resemblance to their direct analogues governing the probability mass function, i.e. the update rules in Eq. \eqref{eq:pmf_recur}.

\begin{proposition}[CDF Recurrence]\label{prop:cdf_recur}
The cumulative distribution function of the classical occupancy problem, $F(k,t,n)$, obeys the recurrence relations 
\begin{equation}
F(k,t+1,n) = 
\begin{dcases}
\frac{n-k}{n} F(k-1,t,n)  +  \frac{k}{n} F(k,t,n) & \textnormal{if } \;  k < t;  \\
1 & \textnormal{if }  \; k \ge t;  
\end{dcases}
\end{equation}
for all $1 \le k \le n$ and $t > 0$.
\end{proposition}
\begin{proof}
Beginning with the latter case ($k \ge t$), when $t$ is smaller or equal to state $k$ no probability mass has flowed further than state $k$, or, put another way, no path has reached a state further than $k$ yet. All probability mass (in total equal to one) is therefore contained in the previous states. 

For the case $k < t$, we proceed as follows:
\begin{itemize}
\item Note that by definition we have (in the state space)
\[
f(k,t,n) = F(k,t,n) - F(k-1,t,n).
\]
\item Furthermore, observe that (in the time domain)
\[
F(k,t+1,n) = F(k,t,n) - \frac{n-k}{n} f(k,t,n),
\]
which after some rearranging becomes
\[
f(k,t,n) = \frac{n}{n-k} \Big( F(k,t,n) - F(k,t+1,n) \Big).
\]
\item Equating the righthand sides of the foregoing two expressions, i.e.
\[
\frac{n}{n-k} \Big( F(k,t,n) - F(k,t+1,n) \Big) = F(k,t,n) - F(k-1,t,n),
\]
and rearranging, we deduce
\[
\frac{n}{n-k} F(k,t+1,n) = \frac{k}{n-k} F(k,t,n) + F(k-1,t,n),
\]
which after multiplication by $\frac{n-k}{n}$ becomes the desideratum.
\end{itemize}
\end{proof}

The complementary c.d.f. obeys a virtually identical system of recurrence relations, which we state in the following proposition.
\begin{proposition}[Complementary CDF Recurrence]\label{prop:ccdf_recur}
The complementary cumulative distribution function of the classical occupancy problem, $\overline{F}(k,t,n)$, obeys the recurrence relations
\begin{equation}
\overline{F}(k,t+1,n) = 
\begin{dcases}
\frac{k}{n} \overline{F}(k,t,n) + \frac{n-k}{n} \overline{F}(k-1,t,n) & \textnormal{if } \;  k \le t;  \\
0 & \textnormal{if }  \; k > t;  
\end{dcases}
\end{equation}
for all $1 \le k < n$ and $t > 0$; with $\overline{F}(0,t,n) = 1$ for $t > 0$.
\end{proposition}
\begin{proof}
The proof is analogous to that of Proposition \ref{prop:cdf_recur}. To summarize:
\begin{itemize}
	\item For $k > t$, observe that no path has reached state $k$ yet, therefore this (cumulative) probability is zero.
	\item For $k \le t$, write the p.m.f. in terms of the complementary c.d.f. in two ways, that is
		\[
		 \overline{F}(k-1,t,n) - \overline{F}(k,t,n) \overset{(1)}{=} f(k,t,n) \overset{(2)}{=} \frac{n}{n-k} \Big(\overline{F}(k,t+1,n) - \overline{F}(k,t,n) \Big) ,
		\]
and rearrange to isolate $\overline{F}(k,t+1,n)$.
\end{itemize}
\end{proof} 

Again, we note the great similarity of this system of recurrence relations to the one describing the p.m.f. in Eq. \eqref{eq:pmf_recur}. Note that the system of recurrence relations for the complementary c.d.f. has the benefit of being sparse -- owing to the second case ($k > t$) being equal to zero.  

\subsubsection{A Matrix Implementation}\label{sec:cocc_matrix_implementation}
The system of linear recurrence relations for the complementary c.d.f. has an equivalent formulation in terms of a bidiagonal matrix -- highly similar to transition matrix $P$ -- which we state next.

Given the striking resemblance of the system of recurrence relations to that of the p.m.f. (Eq. \ref{eq:pmf_recur}), we present Proposition \ref{prop:ccdf_recur} here in matrix form, as follows
\begin{equation}\label{eq:ccdf_recur_matrixform}
C = \dfrac{1}{n}\begin{bmatrix}
n 	& n-1 	&   			&   			&   		&	    			&			&	  \\
  	& 1		& n-2 			&				& 			&	    			&			&	  \\
	&		& \small\ddots 	& \small\ddots 	&			&					& 			&	  \\
	&		&				&	k-1			& 	n-k 	&					&	 	    &     \\ 
	&		&				&				&   k		& 	n-k-1 			&	 	    &     \\ 		
	&		&				&				&			& 	\small\ddots 	& 	\small\ddots	&  	 	  \\ 	
	&		&				&				&			& 		 			&	n-2		&  1	  \\ 
	&		&				&				&			& 		 			&			&  n-1	  \\ 	
\end{bmatrix},
\end{equation}
and formally in the next proposition.
\begin{proposition}\label{prop:ccdf_recur_matrixform}
The complementary distribution function, $\overline{F}(j,t,n)$, for $t>0$, is given by the $j$\textsuperscript{th} element in the top row of the matrix $C^t$. With $C$ defined by
\begin{equation}
\left[C \right]_{i,j}^{} = \begin{cases}
\dfrac{n}{n}     	& \textnormal{if } 	j=i=1; 		\\
\dfrac{j}{n}   		& \textnormal{if }  j=i\neq 0; 	\\
\dfrac{n-j}{n}   	& \textnormal{if } 	j=i+1; 		\\
0					& \textnormal{otherwise};
\end{cases}
\end{equation}
and indexed $0 \le {i,j} < n$.
\end{proposition}
\begin{proof} 
This is a direct matrix translation of \Cref{prop:ccdf_recur}; with a proof identical to that of \Cref{prop:pbp_ccdf_recur_matrixform} (\Cref{chpt:pbp_in_disc_time}).
\end{proof}
\Cref{prop:ccdf_recur_matrixform} provides a way to calculate $\overline{F}(k,t,n)$ as the topmost row of $C^t$. Note, that if one is only interested in the complementary c.d.f. of the first $k$ states, then $C$ may be constrained to its top-left $(k+1) \times (k+1)$ submatrix. Furthermore, note that exponentiation of $C$ has the exact same computational complexity as exponentiation of transition matrix $P$. 

\section{Expectation and Variance}\label{sec:cocc_expectation_variance}
In this section we derive the expectation and variance of the classical occupancy distribution. First, we briefly recount a commonly seen derivation which is based on a sum of (non-independent) Bernoulli random variables. We shall contrast it to the approach taken by \cite{price1946multinomial}.

Suppose, we make $t$ draws with replacement from a population $[n]$. The probability that any element, say element $i$, is not sampled in $t$ draws equals $(1 - \frac{1}{n})^t$, by independence of the draws. Define the Bernoulli random variable $B_i$ as being equal to $1$ if element $i$ is not sampled, and equal to $0$ if the element $i$ is sampled in any of the $t$ draws. Then the expectation of the sum of these Bernoulli variables, $\E\left[\sum_{i=1}^n B_i\right] = \sum_{i=1}^n \E[B_i] = n (1 - \frac{1}{n})^t $, equals the expected number of non-sampled population elements, $\E[Y_t] = n - \E[X_t] $. 

We find this derivation unsatisfying because of its indirectness and disregard for the true underlying distribution. Fortunately, \citet[Sec. 5]{price1946multinomial} provides a simple and direct derivation of the expectation and variance of the classical occupancy distribution, that we wish to emulate here.

\begin{proposition} The expectation of the random variable $X_t$ having the classical occupancy distribution (of \Cref{prop:pmf_eigen}) with parameters $n,t$ equals
\begin{equation} \label{eq:expectation_classical_occupancy_price}
\E\left[X_t \,\vert\, n \right] = n - n \left(1 - \frac{1}{n} \right)^t.
\end{equation}
\end{proposition}
\begin{proof}
Let us introduce the notation $\Pr(X_t = k  \,\vert\, n)$ to denote the probability that the process $X_t$ over population $[n]$ is in state $k$ at time $t$. The proof relies on a comparison of $\Pr(X_t = k  \,\vert\, n)$ with $\Pr(\widetilde{X}_t = k  \,\vert\, n-1)$, in combination with $\sum_{k=0}^n \Pr(X_t = k  \,\vert\, n) = 1$.

From equation \eqref{eq:pmf_eigen2} for the probability mass function of the classical occupancy distribution, we have
\begin{equation}
\Pr(X_t = k  \,\vert\, n) = \frac{(n)_{k}^{} }{n^t} \stirlingtwo{t}{k}, \quad \text{and} \quad \Pr(\widetilde{X}_t = k  \,\vert\, n-1) = \frac{(n-1)_{k}^{} }{(n-1)^t} \stirlingtwo{t}{k}.
\end{equation}
Note the Stirling number as a common factor in both expressions. Therefore, we have 
\begin{equation}
\begin{aligned}
\phantom{\therefore\;}  \frac{(n-1)^t }{ (n-1)_{k}}\Pr(\widetilde{X}_t = k  \,\vert\, n-1) 	\;\;=&\;\;\;     \frac{n^t }{(n)_{k} } \Pr(X_t = k  \,\vert\, n)	\\
\therefore\; (1 - \frac{1}{n})^t \Pr(\widetilde{X}_t = k  \,\vert\, n-1)  				  	\;\;=&\;\;\;    	\frac{ (n-1)_{k}}{(n)_{k} } \Pr(X_t = k  \,\vert\, n)	\\
\therefore\;  (1 - \frac{1}{n})^t \Pr(\widetilde{X}_t = k  \,\vert\, n-1) 		\;\;=&\;\;\;       \frac{ n-k}{n } \Pr(X_t = k  \,\vert\, n)		\\
\therefore\;  n(1 - \frac{1}{n})^t \Pr(\widetilde{X}_t = k  \,\vert\, n-1) 	\;\;=&\;\;\;   (n-k) \Pr(X_t = k  \,\vert\, n).
\end{aligned}
\end{equation}
Then sum both sides over $k$: 
\begin{itemize}
\item For the right-hand side, we obtain
\begin{equation}
\begin{aligned}
\sum_{k=0}^{n}(n-k) \Pr(X_t = k  \,\vert\, n) &=  n\sum_{k=0}^{n} \Pr(X_t = k  \,\vert\, n) - \sum_{k=0}^{n} k\Pr(X_t = k  \,\vert\, n) \\ 
												&= n - \E\left[X_t \, \vert \, n \right]. \\
\end{aligned}
\end{equation}
\item And for the left-hand side, we have
\begin{equation}
n(1-\frac{1}{n})^t \sum_{k=0}^{n-1} \Pr(\widetilde{X}_t = k \,\vert\, n-1) = n(1-\frac{1}{n})^t.
\end{equation}
\end{itemize}
Solving for the expectation, we obtain $\E\left[X_t \, \vert \, n \right] = n - n(1-\frac{1}{n})^t$, as desired.
\end{proof}

The variance can be obtained along the same lines, that is from 
\begin{equation}
{k(n-k)\Pr(X_t = k  \,\vert\, n) =  k n (1 - \frac{1}{n})^t \Pr(\widetilde{X}_t = k  \,\vert\, n-1)},
\end{equation}
which after some manipulation, yields the expression
\begin{equation}\label{eq:variance_classical_occupancy}
\begin{aligned}
\V \left[ X_t  \, \vert \, n \right] &=  n\left(1-\frac{1}{n}   \right)^t  + n(n-1) \left(1 - \frac{2}{n} \right)^t - n^2 \left(1 - \frac{1}{n} \right)^{2t}.
\end{aligned}
\end{equation}
We note that the approach suggested by \cite{price1946multinomial} can be used to recursively determine any moment in terms of lower order moments. It should also be mentioned that \cite{weiss1958limit_distri} derived the raw moments of all orders, albeit in an algebraically more involved manner.

\section{The Randomized Occupancy Model}\label{sec:rand_occ_model}
There are many variations of the occupancy model in the literature \citep[see e.g.][]{feller1968ed3, harkness1969occupancy, kotz_johnson1977urn_models, holst1986birthday}. A common variation in the applied literature is the addition of a single and constant retention probability, $p$, for all urns. The following definition paraphrases \cite{harkness1969occupancy}:

\begin{definition}[Randomized Occupancy]\label{def:randomized_occupancy}
\textnormal{Given $n$ distinguishable urns of unbounded capacity, and $t$ throws with distinguishable balls having equiprobability of ending up in any of the $n$ urns, and each ball having probability $p$ of occupying the urn, what is the probability of finding $k$ urns occupied?}
\end{definition}
That is, each ball has probability $1-p$ of ``missing" the targeted urn, or ``falling through" the urn. This model is a generalization of the classical occupancy problem, which itself is recovered as a special case when $p=1$.

In contrast to \cite{harkness1969occupancy}, we describe the model not in terms of empty urns, $Y_t$, but in terms of occupied urns, $X_t = n - Y_t$. We go a step further by simplifying the expressions for the probability mass function and distribution function, and \textit{en passant} answer a question posed by \citet[Sec. 4]{harkness1969occupancy} in his section on the Markovian representation of the problem. 

\subsubsection{Markov Chain Embedding} 
In the randomized occupancy model, the discrete-time Markov process $X_t$ transitions from state $k$ to state $k+1$ when an unoccupied urn is randomly targeted \emph{and} the ball is retained by the urn. These events are independent, hence their joint probability is
\begin{equation}\label{eq:randomized_occupancy_transition_prob}
p_{k}^{} = \frac{n-k}{n} p.
\end{equation}
Equivalently, the stop probability in state $k$ is
\begin{equation}\label{eq:randomized_occupancy_stop_prob}
q_{k}^{} = 1 - p_{k}^{} = 1 - \frac{n-k}{n} p = \frac{n - (n-k)p}{n}\,. 
\end{equation}
Formally, we define the transition matrix as follows
\begin{equation}\label{eq:randomized_occupancy_trans_matrix}
\left[ P \right]_{i,j}^{} =  
\begin{dcases} 
q_i = \frac{n - (n-j)p}{n}  	& \mbox{if }\, j = i; 	\\ 
p_i = \frac{(n - j)p}{n} 		& \mbox{if }\, j = i + 1;	\\
0    	        				& \mbox{otherwise};
\end{dcases}
\end{equation}
for all integer indices $0\le i,j \le n$. These relations may equivalently be expressed in terms of the following recurrence relations:
\begin{equation}\label{eq:randomized_occupancy_pmf_recur}
p_{i,j}^{(t+1)} \overset{\text{def}}{=} \left[ P^{t+1} \right]_{i,j}^{} = 
\begin{dcases} 
(1-p)^{t+1}													& \mbox{if }\; 		i=j=0; 				\\
q_{j}^{} p_{i,j}^{(t)}  \,+\, p_{j-1}^{} p_{i,j-1}^{(t)} 	& \mbox{if }\;  	0 < i \le j \le n \text{ and } j < t; 	\\ 
0 															& \mbox{otherwise}; 				\\ 
\end{dcases}
\end{equation}
for $i,j \in S = \{0, 1,\, \dots, n\}$  and $t > 0$; and initially $p_{i,i}^{(0)} = 1$, $i \in S$.

\subsubsection{Representation as Mixture}\label{sec:rand_occ_as_mixture}
\citet[Sec. 3.3]{kotz_johnson1977urn_models} express the randomized occupancy distribution as a binomial mixture of \emph{classical} occupancy distributions. Taking this as our starting point, we obtain the following expression for the probability mass function
\begin{equation}\label{eq:randomized_occupancy_pmf}
\begin{aligned}
f(k,t;p,n) = p_{0,k}^{(t)}  &= \binom{n}{k} \sum_{j=0}^k (-1)^{k-j} \binom{k}{j} \left(1 - p + \frac{jp}{n}  \right)^t & \text{(Johnson \& Kotz)} \\
							&= \binom{n}{k} \sum_{j=0}^k (-1)^{k-j} \binom{k}{j} \left(\frac{n - (n - j)p}{n}  \right)^t &  \\
							&= \binom{n}{k} \sum_{j=0}^k (-1)^{k-j} \binom{k}{j} q_{j}^t. & \text{(by Eq. \eqref{eq:randomized_occupancy_stop_prob})}
\end{aligned}
\end{equation}
The cumulative distribution function can be obtained by means of the first hitting time formulation for state $k$, that is by using the methods of \Cref{sec:pbp_first_hit_times}, in particular relations \eqref{eq:pbp_cdf_phase_equiv_stage} and \eqref{eq:pbp_ccdf_first_hit_time_as_finite_sum_of_stage_pmf}, as we demonstrate in the following proposition.
\begin{proposition}[CDF Randomized Occupancy]\label{prop:randomized_occupancy_cdf}
The cumulative distribution function of the randomized occupancy problem is equal to
\begin{equation}\label{eq:randomized_occupancy_cdf}
F(k,t;p,n) = (n-k) \binom{n}{k} \sum_{j=0}^k (-1)^{k-j} \binom{k}{j} \dfrac{q_{j}^t}{(n-j)}  
\end{equation}
for all integers $0 \le k \le n$,\, $t > 0$, and any real $0 < p \le 1$.
\end{proposition}

\begin{proof}
We have
\begin{equation}
\begin{aligned}
F(\underline{k},t; p, n)  &= p_{k}^{} \sum_{d=t}^\infty f(k,\underline{d}; p,n) & \text{(as hitting time)} \\
			&= p_{k}^{} \sum_{d=t}^\infty \binom{n}{k} \sum_{j=0}^k (-1)^{k-j} \binom{k}{j} q_{j}^d 	& \text{(by Eq. \eqref{eq:randomized_occupancy_pmf})}				\\
			&= p_{k}^{} \binom{n}{k} \sum_{j=0}^k (-1)^{k-j} \binom{k}{j} \sum_{d=t}^\infty q_{j}^d  & \text{(reorder summation)}	\\
			&= p_{k}^{} \binom{n}{k} \sum_{j=0}^k (-1)^{k-j} \binom{k}{j} \frac{q_{j}^t}{1 - q_{j}^{}} & \text{(geometric series)}	\\
			&= p \frac{(n-k)}{n} \binom{n}{k} \sum_{j=0}^k (-1)^{k-j} \binom{k}{j} \frac{q_{j}^t}{p_{j}^{}} 	& \text{(simplify)} \\
			&= (n-k) \frac{\cancel{p}}{\cancel{n}} \binom{n}{k} \sum_{j=0}^k (-1)^{k-j} \binom{k}{j} \dfrac{\cancel{n}}{\cancel{p}} \dfrac{q_{j}^t}{(n-j)}  	\\			
			&= (n-k) \binom{n}{k} \sum_{j=0}^k (-1)^{k-j} \binom{k}{j} \dfrac{q_{j}^t}{(n-j)},  		\\
\end{aligned}
\end{equation}
as desired.
\end{proof}
We note the great similarity between \Cref{prop:randomized_occupancy_cdf} and \Cref{prop:cdf_from_eigendecomp}. In fact, Eq. \eqref{eq:cdf_from_eigendecomp} can be recovered from Eq. \eqref{eq:randomized_occupancy_cdf} when we set $p=1$, that is: $q_{j}^{} = \sfrac{j}{n}$.

\Cref{prop:randomized_occupancy_cdf} can also be derived from the eigendecomposition of the transition matrix $P$. Or, \emph{vice versa}, the eigenvalues and eigenvectors are strongly suggested by equation \eqref{eq:randomized_occupancy_pmf} (and \Cref{prop:randomized_occupancy_cdf}). We record this assertion as an exact statement in the following proposition.
\begin{proposition}\label{prop:rand_occ_eigendecomp}
The eigendecomposition of the transition matrix of the randomized occupancy problem with parameters $t,p,n$, is given by 
\begin{equation*}
P^t = U \Lambda^t U^{-1}
\end{equation*}
with 
\begin{equation*}
\Lambda = \dfrac{1}{n} \diag\left(1-p,\, n - (n-1)p,\, n - (n-2)p,\, \dots ,\, n - p,\, n\right)
\end{equation*}
and
\begin{equation*}
[U]_{{i,j}}^{} = 
\begin{cases}
\binom{n-i}{n-j} & \textnormal{if }\, i \le j; \\
0                & \textnormal{otherwise};
\end{cases}
\quad
\text{ } \quad 
[U^{-1}]_{{i,j}}^{} = 
\begin{cases}
(-1)^{j-i}\binom{n-i}{n-j} & \textnormal{if }\, i \le j; \\
0                          & \textnormal{otherwise}.
\end{cases}
\end{equation*}
\end{proposition}
\begin{proof}
In the same spirit as equation \eqref{eq:transition_mat_eigenvals} and Lemmas \ref{lem:eigenvector_mat_U} and \ref{lem:eigenvector_mat_U_inverse} of \Cref{sec:cocc_markov_chain}.
\end{proof}
Note that the eigenvalues are uniquely determined, but the eigenvectors are unique only upto nonzero scalar multiplication.

\Cref{prop:rand_occ_eigendecomp} directly answers the remark of \citet[Sec. 4]{harkness1969occupancy} that \textit{``Explicit `closed form' expressions for these sums of a simpler nature do not seem to be possible to obtain."} (Harkness then continues with a discussion of approximations.) However, we obtain the following explicit formula for said probabilities
\begin{equation}\label{eq:cond_pmf_randomized_occ_harkness}
\Pr(X_t = k \,\vert\, X_0 = r) = p_{r,k}^{(t)} = \binom{n-r}{n-k} \sum_{j=0}^{k-r} (-1)^{k-r-j} \binom{k-r}{j} q_{j+r}^t,
\end{equation}
for integer $0\le r \le k \le \min(n,t)$; and $0$ otherwise.\footnote{The reader may want to compare Eq. \eqref{eq:cond_pmf_randomized_occ_harkness} to  Eq. \eqref{eq:explicit_from_eigendecomp} of \Cref{prop:transition_probs_classical_occupancy}.} 

\subsubsection{Recurrence Relation for Distribution Function}
We conclude with a simple system of recurrence relations for the complementary c.d.f., just as we did before for the classical occupancy distribution.

\begin{proposition}[Complementary CDF: Recurrence]\label{prop:radomized_occ_ccdf_recur}
The complementary cumulative distribution function of the randomized occupancy problem, $\overline{F}(k,t;p,n)$, obeys the recurrence relations 
\begin{equation}
\overline{F}(k,t+1;p,n) = 
\begin{dcases}
\frac{(n-k)p}{n}\overline{F}(k-1,t;p,n) + \frac{n - (n-k)p}{n} \overline{F}(k,t;p,n) & \textnormal{if } \;  k \le t;  \\
0 & \textnormal{if }  \; k > t.  
\end{dcases}
\end{equation}
for all $0 \le k < n$, and $t > 0$.
\end{proposition}
\begin{proof}
Proof is identical to \Cref{prop:ccdf_recur} and \Cref{prop:pbp_ccdf_recur}, but nonetheless included in Appendix \ref{apx:proofs}.

\end{proof}
These recurrence relations are of interest because they provide a simple and efficient way to calculate the tail probabilities of the randomized occupancy distribution. 

\subsection{Factorial and other Moments}
In closing, we record some results from \citet[Sec. 3]{harkness1969occupancy} pertaining to the the raw and factorial moments of the randomized occupancy model, thereby generalizing those that \cite{weiss1958limit_distri} obtained for the classical occupancy distribution. 

Harkness derived the factorial moments of the process $Y_t \overset{\text{def}}{=} n - X_t$ using generating functions, to arrive at
\begin{equation}
\E\left[ \left( Y_t \right)_k^{} \right] = (n)_k^{}  \left( 1 - \frac{kp}{n} \right)^t,
\end{equation}
in polished form. The raw moments then follow directly as 
\begin{equation}
\E [  Y_t \,^k ] =	\sum_{j=0}^k  \stirlingtwo{k}{j} (n)_k^{} \left( 1 - \frac{kp}{n} \right)^t,
\end{equation}
by slight adaptation of \citet[Eq. 10]{harkness1969occupancy}, or from identity \eqref{eq:power_is_sum_fallingfactorial_times_stirlingtwo}. Hence, the expectation and variance of the randomized occupancy distribution are given by: 
\begin{equation}
\begin{aligned}
\E\left[ Y_t \right] &= n  \left( 1 - \frac{p}{n} \right)^t, \quad  \E\left[ X_t \right] = n - n  \left( 1 - \frac{p}{n} \right)^t, \\
\V\left[ Y_t \right] &= \V\left[ X_t \right] = n \left(1- \frac{p}{n} \right)^t + n (n-1) \left(1- \frac{2p}{n} \right)^t - n^2 \left(1- \frac{p}{n} \right)^{2t}.
\end{aligned}
\end{equation}
Note that when we set $p=1$, we obtain the expectation and variance of the classical occupancy distribution, i.e. expressions \eqref{eq:expectation_classical_occupancy_price} and \eqref{eq:variance_classical_occupancy}.

\section{Chapter Summary}
In \Cref{chpt:special_case} we refined all results of \Cref{chpt:pbp_in_disc_time} for the specific parametrization of the classical (and randomized) occupancy problem(s). In contrast to the approach taken in \Cref{chpt:pbp_in_disc_time}, which employed the complete homogeneous symmetric polynomial, the main tool in \Cref{chpt:special_case} was the eigendecomposition of the transition matrix. From the eigendecomposition we derived all transition probabilities in explicit form, in terms of $r$-Stirling numbers of the second kind, and as appropriately attuned finite differences. 
In doing so, we generalized the classic and well known expressions for the probability mass function and distribution function to their initial-state conditioned versions in Propositions \ref{prop:transition_probs_classical_occupancy} and \ref{prop:cdf_from_eigendecomp}, respectively.

Furthermore, in \Cref{prop:ccdf_recur} we derived a completely new sparse linear system of recurrence relations for the complementary distribution function. We provided its matrix implementation in \Cref{prop:ccdf_recur_matrixform}.\footnote{Propositions \ref{prop:ccdf_recur} and \ref{prop:ccdf_recur_matrixform} are the natural adaptations of Propositions \ref{prop:pbp_ccdf_recur} and \ref{prop:pbp_ccdf_recur_matrixform} from \Cref{chpt:pbp_in_disc_time}.}

We treated the randomized occupancy model of \cite{harkness1969occupancy}, and showed that the Markovian perspective has relevance for that problem as well. Most evidently by answering a question raised by Harkness as to whether it is possible to simplify the expressions for the elements in the exponentiated transition matrix, $p_{r,k}^{(t)}$. We were able to answer this question affirmatively in Eq. \eqref{eq:cond_pmf_randomized_occ_harkness}.

Finally, for completeness and to unify the vast but scattered literature, the respective moments of the classical and randomized occupancy distribution were discussed. In particular, for the classical occupancy model we provided the elegant derivation of \cite{price1946multinomial}.

Appendix \ref{apx:matrix_examples} contains concrete examples of the matrices $P,\,U,\,U^{-1},\,U^{-1}\Sigma$ and $C$.

\paragraph{Discussion.}
Even though all results of this chapter can be derived using the combinatorial approach of \Cref{chpt:pbp_in_disc_time}, we took a contrasting matrix-algebraic approach here to illustrate that more than one road leads to Rome. 

We thus demonstrated the capabilities of the matrix-algebraic Markovian approach. This approach leads naturally to a generalization of the classic expressions, namely to their initial-state conditioned versions. Thereby generalizing and extending the work of \cite{uppuluri1971occupancy}, \cite{kotz_johnson1977urn_models}, \cite{harkness1969occupancy}, and \cite{feller1968ed3} on this problem.

The simple system of recurrence relations for the distribution function will perhaps find greatest practical use, for example to obtain fast and precise tail probabilities. Apart from its practicality, its matrix implementation in particular takes a strikingly simple and elegant\footnote{Caveat lector: beauty is in the eye of the beholder.} form in \Cref{prop:ccdf_recur_matrixform}.

\chapter{Conclusion and Discussion}\label{chpt:conclusion_discussion}
Our primary interest in this thesis was a property of sampling with replacement, namely: the distribution of the number of distinct elements one obtains after sampling a finite population a given number of times with replacement. This problem is known as the classical occupancy problem.

We approached this problem by asking what a purely Markovian perspective on it would yield in terms of greater understanding and new results. By Markovian perspective we mean the formulation of the problem as a discrete-time Markov chain.

We took this route because (\textit{1}) the Markovian structure is apparent, but (\textit{2}) the only paper taking this position as a starting point, \cite{uppuluri1971occupancy}, did not exhaust this vantage point to its full potential.

Given our Markovian starting position, we advanced in three directions.

\paragraph{Probability Mass Function.}
We derived the known classic expressions for the probability mass function in a new manner using the eigendecomposition of the transition matrix. We generalized these expressions to their initial-state conditioned versions, expressing these in terms of $r$-Stirling numbers, and as finite differences. We are not aware of these latter results having been obtained before.

\paragraph{Cumulative Distribution Function.}
The cumulative distribution function has historically received less attention than the probability mass function. We instead made it a focal point.

Using the eigendecomposition, we derived the distribution function immediately in its general, initial-state conditioned form. Which we compactly expressed as an appropriately evaluated finite difference. 

However, our main (practical) novelty is the system of recurrence relations that we derived for the complementary distribution function in \Cref{prop:ccdf_recur}, and its matrix form in \Cref{prop:ccdf_recur_matrixform}. With this sparse linear system, exact computation of tail probabilities for large $n$ and $t$ becomes feasible. 

We believe that these results constitute an appreciable advance in the current knowledge of the classical occupancy problem.
   
\paragraph{Generalization to Pure Birth Processes.}   
\Cref{chpt:pbp_in_disc_time} is entirely concerned with the generalization of these results to general discrete-time pure birth processes.

To understand its relation to the literature, we first precisely described the dual of the state distribution, that is in terms of each state's first hitting time. Crucially, the first hitting times can be represented as a convolution of geometric distributions.

\cite{sen1999convolution_geoms} derived the exact form of the probability mass function of this convolution for distinct geometric distributions. In the same context, \citet[Sec. 3]{chen2016conv_negbinom} generalized this result to the convolution of generic geometric distributions, removing the distinctness condition. 

We gave a new, elementary proof of both results based on a combinatorial path-counting argument combined with the complete homogeneous symmetric polynomial. 

We then extended the results of the aforementioned authors to the distribution function, for both generic and distinct transition probabilities in \Cref{thm:pbp_general_ccdf} and \Cref{prop:pbp_cdf_distinct}, respectively.

Finally, we proved in \Cref{prop:pbp_ccdf_recur} that the simple system of recurrence relations for the classical occupancy distribution holds \textit{mutatis mutandis} for the entire class of discrete-time pure birth distributions. We consider this to be a notable result. \\

\noindent We thus demonstrated that a Markovian viewpoint combined with the right mathematical instruments -- the complete homogeneous symmetric polynomial and the eigendecomposition of the transition matrix -- yields new results for a classic problem of probability theory. 

In general, we hope that our work has contributed to a better understanding of sampling with replacement and pure birth processes in discrete time. 
Additionally, we hope that our treatment has unified part of the vast but scattered literature on the exact distribution of the classical occupancy problem, and the diverse tools that have been proposed by various authors. \\  

\noindent Concluding this thesis, we wish to suggest some directions for future research. 

First, much of this work caries over to continuous-time pure birth processes. By considering a convolution of exponential distributions instead of geometric distributions, the Markov embedding becomes one in continuous time. One could ask how much of the elementary nature of our approach remains intact in continuous time.

Second, the description of the process $\mathbf{X}$ in terms of the evolution of its stages, has a dual in the collection of each state's first hitting time, say $\mathbf{T}$. We have fruitfully exploited that duality in this thesis, most evidently in the case of the distribution function. It would be interesting to see if this duality can be further exploited. For example, in terms of their moments. 

\newpage 

\phantomsection

\addcontentsline{toc}{chapter}{Bibliography}

\bibliography{thesis_bibliography}

\newpage
\chapter{Appendix}\label{chpt:apx}
\epigraph{\flushright{\emph{``All play means something."}}}{J. Huizinga}

\section{Problem Variations}\label{apx:other_variations}
The problem variations in this section arose mainly out of curiosity, but nonetheless provide a way of validating the chosen approach of this thesis. The treatment is quicker and the tone is less formal than the main text.

\subsection{The Complementary Occupancy Chain: Interchange of Transition and Stop Probabilities}
This section details the problem variation in which the stop and transition probabilities of the classical occupancy problem are interchanged.

\begin{figure}[htb]
\centering
\begin{tikzpicture}[node distance = 2.1cm,  auto] 

\node[state] 						(s1) 	{1};
\node[state, right of=s1] 			(s2) 	{2};

\node[draw = none, right of = s2] 	(s4_) 	{$\cdots$};
\node[state, right of=s4_] 			(s5) 	{$k$};
\node[state, right of=s5] 			(s6) 	{$k{+}1$};
\node[draw = none, right of = s6] 	(s7_) 	{$\cdots$}; 
\node[state, right of=s7_] 			(s8) 	{$n{-}1$};
\node[state, thick, right of=s8] 	(s9) 	{$n$};

\draw (s1) 	edge[loop below] node {$\frac{n-1}{n}$}  	(s1);
\draw (s2) 	edge[loop below] node {$\frac{n-2}{n}$} 	(s2);

\draw (s5)	edge[loop below] node {$\frac{n-k}{n}$}		(s5); 
\draw (s6) 	edge[loop below] node {$\frac{n-k-1}{n}$}	(s6); 

\draw (s8) 	edge[loop below] node {$\frac{1}{n}$} (s8);
\draw (s9) 	edge[loop below, thick] node {$1$}			 (s9);

\draw[->]  (s1) 	to node [midway, above] { $\frac{1}{n}$ } 	  	(s2);
\draw[->]  (s2) 	to node [midway, above] { $\frac{2}{n}$ } 	  	(s4_);

\draw[->]  (s4_) 	to node [midway, above] { $\frac{k-1}{n}$ }  	(s5);
\draw[->]  (s5) 	to node  				{ $\frac{k}{n}$ }    	(s6);  
\draw[->]  (s6) 	to node [midway, above] { $\frac{k+1}{n}$ }  	(s7_); 
\draw[->]  (s7_)	to node [midway, above] { $\frac{n-2}{n}$ }  	(s8); 
\draw[->]  (s8)		to node [midway, above] { $\frac{n-1}{n}$ }  	(s9);

\end{tikzpicture}
\caption{Transition diagram of the discrete-time Markov process $\widehat{X}_t$.}
\label{fig:compl_occpancy_markov_chain}
\end{figure}
Let $\widehat{X}_t$ denote the process defined by the Markov chain in \Cref{fig:compl_occpancy_markov_chain}. The corresponding transition matrix of the process $\widehat{X}_t$ equals
\begin{equation}\label{eq:matrix:1-step-transition_interchange_probs}
\begin{aligned} \widehat{P}   &=  \dfrac{1}{n} 
\begin{bmatrix*}
n-1 & 1            &                       &  	 					&   	        & 						&  \\
  &  n-2           & 2                     &  	 				    &   		    & 						&  \\
  &				   &    \small\ddots       &     \small\ddots       &   			&  						&  \\
  &                &					   &         n-k            &      k        &    			        &  \\
  &				   & 					   &						&		\small\ddots  & 		\small\ddots   	&   \\
  &				   &      				   &						&				&	1    	            & n-1 \\
  &     		   &      				   &						&				&	    	            & n  \\
\end{bmatrix*}.
\end{aligned}
\end{equation}
Formally, we record the transition matrix $\widehat{P}$ as
\begin{equation}\label{eq:compl_chain_def_transition_probs_formal}
[\widehat{P}]_{i,j}^{} = \hat{p}_{i,j}^{} = 
\begin{dcases} 
\frac{n-i}{n}   & \mbox{if }\, j = i \neq n;  \\
\frac{n}{n}   	& \mbox{if }\, j = i = n; 	  \\ 
\frac{i}{n}     & \mbox{if }\, j = i + 1; 	  \\ 
0    			& \mbox{otherwise};
\end{dcases}
\end{equation}
indexed as $1 \le i,j \le n$.

The $r$-Stirling numbers of the second kind are indispensable in this case to express the p.m.f. and c.d.f. in closed form. 

Again, treating the probability weight contributed by the transitions and the stops separately, and writing the latter's contribution in terms of $r$-Stirling numbers of the second kind, we have
\begin{equation}\label{eq:dual_pmf_stirling2_matrix_elements}
\begin{aligned}
\Pr(\widehat{X}_t = k \,\vert\, \widehat{X}_0 = r) &= \frac{1}{n^t} \left(\prod_{j=r}^{k-1} j \right) h_{t-k+r}(n-r,n-r-1,...,n-k) \\ 
												   &= \frac{1}{n^t} \frac{(k-1)!}{(r-1)!} \stirlingtwo{t - k + n}{n-k,\; n-r},
\end{aligned}
\end{equation} 
for $1 \le r \le k \le n$, $t > 0$; and $0$ otherwise. 

We obtain the unconditional p.m.f. straightforwardly from \eqref{eq:dual_pmf_stirling2_matrix_elements} by setting $r=1$. The expression becomes
\begin{equation}\label{eq:compl_occupancy_pmf_formula}
\begin{aligned}
\hat{f}(k,t,n)  	&= \frac{(k - 1)!}{n^t} h_{t-k+1}(n{-}1, n{-}2, ..., n{-}k) \\
		 			&= \frac{(k - 1)!}{n^t} \stirlingtwo{t-k+n}{n-k,\; n-1},
\end{aligned}
\end{equation}
for $1 \le k \le n$ and $t > 0$.

The complementary c.d.f. can be obtained by truncating the chain in \Cref{fig:compl_occpancy_markov_chain} by making state $k+1$ absorbing, the resulting expression for the unconditional c.d.f. ($\widehat{X}_0 = 1$) then becomes
\begin{equation}\label{eq:compl_occupancy_ccdf_formula}
\begin{aligned}
\overline{\widehat{F}}(k,t,n) 	& \overset{\text{def}}{=} \Pr(\widehat{X}_t > k \,\vert\, \widehat{X}_0 = r) 	\\
							  	&= \frac{k!}{n^t} h_{t-k}(n, n{-}1, ..., n{-}k) \\
							 	&= \frac{k!}{n^t} \stirlingtwo{t-k+n}{n-k,\; n}, 
\end{aligned}
\end{equation}
for $1 \le k < n$ and $t > 0$. The c.d.f. follows from $\widehat{F}(k,t,n) = 1 - \overline{\widehat{F}}(k,t,n)$. 

Furthermore, note the following relation:
\begin{equation}
n^t \, \overline{\widehat{F}}(k,t,n) = k! h_{t-k}(n, n{-}1, n{-}2, ..., n{-}k) = (n+1)^{t}\, \hat{f}(k+1,t,n+1),
\end{equation}
which when summed over $k$, can conveniently be used to derive the expectation of $\widehat{X}_t$ (cf. the method of \cite{price1946multinomial} in \Cref{sec:cocc_expectation_variance}).

Finally, note, with regard to the absorption probabilities in the rightmost column of the transition matrix $\widehat{P}^t$, that the elements in the second to penultimate rows are \emph{not expressible} straightforwardly in terms of $r$-Stirling numbers. To be precise, those numbers are $\hat{p}_{i,n}^{(t)} = [\widehat{P}^t]_{i,n}^{}$ with $i=2,3,...,n-1$. (These terms miss at least one integer for them to form a sequence of \emph{consecutive integers} in their numerators.)

We now proceed with a short treatment of some recurrence relations for these absorption probabilities, before concluding this section with a few words about the chain's first hitting times.

\subsubsection{Recurrence Relations for the Absorption Probabilities}
In this subsection we derive some simple recurrence relations for the elements in the rightmost column of $\widehat{P}^t$, that is for probabilities $\hat{p}_{k,n}^{(t)}$ with $k = 2, 3, ..., n-1$, as these are not trivial. \newline

\noindent We begin by noting that $\hat{p}_{1,n}^{(t+1)}$ can be expressed in terms of $\hat{p}_{1,n}^{(t)}$ and $\hat{p}_{2,n}^{(t)}$ as follows
\begin{equation}
\hat{p}_{1,n}^{(t+1)} = \hat{p}_{1,1}^{\phantom{()}} \hat{p}_{1,n}^{(t)} \,+\, \hat{p}_{1,2}^{\phantom{()}} \hat{p}_{2,n}^{(t)}.
\end{equation}
The left hand side can be also be expressed as
\begin{equation}
\hat{p}_{1,n}^{(t+1)} =  \hat{p}_{1,n}^{(t)} \,+\, \hat{p}_{n-1,n}^{} \hat{p}_{1,n-1}^{(t)}.
\end{equation}
Equating these expressions and isolating $\hat{p}_{2,n}^{(t)}$, we obtain the expression 
\begin{equation}
\begin{aligned}
\hat{p}_{2,n}^{(t)}	
                    &= \hat{p}_{1,n}^{(t)} \,+\,  (n-1) \hat{p}_{1,n-1}^{(t)},	
\end{aligned}
\end{equation}
after some manipulation. This reasoning generalizes to all states ${k = 2, 3, ..., n}$. Hence the following lemma.
\begin{lemma}\label{lem:recur_k_dual_chain} Denote by $\hat{p}_{i,j}$ the transition probabilities of the Markov chain defined in equation \eqref{eq:compl_chain_def_transition_probs_formal}; then the recurrence relation,
\begin{equation}\label{eq:dual_general_recur_in_statement}
\hat{p}_{k,n}^{(t)} = \hat{p}_{k-1,n}^{(t)} + \frac{n-1}{k-1} \hat{p}_{k-1,n-1}^{(t)}
\end{equation}
holds true for all $k = 2,3,...,n$, and $t \ge 1$.
\end{lemma}
\begin{proof} Take any $k \in \{2,3,..., n\}$. 
\begin{itemize}

\item On the one hand, we have
\begin{equation}
\hat{p}_{k-1,n}^{(t+1)}	= \hat{p}_{k-1,k-1}^{\phantom{()}} \hat{p}_{k-1,n}^{(t)} + \hat{p}_{k-1,k}^{\phantom{()}} \hat{p}_{k,n}^{(t)}.
\end{equation}
\item On the other hand, it holds true that
\begin{equation}
\hat{p}_{k-1,n}^{(t+1)} = \hat{p}_{k-1,n}^{(t)} + \hat{p}_{n-1, n}^{\phantom{()}} \hat{p}_{k-1,n-1}^{(t)}.
\end{equation}
\item Hence, when we equate the right-hand sides of these expressions and isolate $\hat{p}_{k,n}^{(t)}$, the desideratum is obtained as follows
\begin{equation}\label{eq:dual_general_recur_in_proof}
\begin{aligned}
\hat{p}_{k,n}^{(t)} &=  \frac{  ( 1 - \hat{p}_{k-1,k-1}^{} ) \hat{p}_{k-1,n}^{(t)} + \hat{p}_{n-1, n}^{} \hat{p}_{k-1,n-1}^{(t)}  } {{\hat{p}_{k-1,k}^{}}}    	\\
				    &=  \frac{  \hat{p}_{k-1,k}^{} \hat{p}_{k-1,n}^{(t)} + \hat{p}_{n-1, n}^{} \hat{p}_{k-1,n-1}^{(t)}  } {{\hat{p}_{k-1,k}^{}}} \\
				    &=  \hat{p}_{k-1,n}^{(t)}   +  \frac{\hat{p}_{n-1, n}^{}}{\hat{p}_{k-1,k}^{}} \hat{p}_{k-1,n-1}^{(t)}  \\
				    &=  \hat{p}_{k-1,n}^{(t)} + \frac{n-1}{k-1} \hat{p}_{k-1,n-1}^{(t)}.	    
\end{aligned}
\end{equation}
\end{itemize}
\end{proof}
Lemma~\autoref{lem:recur_k_dual_chain} provides a simple recurrence relation, which expresses the conditional absorption probability, $\hat{p}_{k,n}^{(t)}$, in terms of $\hat{p}_{k-1,n}^{(t)}$, and a a probability that can be readily read-off from equation \eqref{eq:dual_pmf_stirling2_matrix_elements}, $\hat{p}_{k-1,n-1}^{(t)}$. 

We may alternatively express this recurrence relation as in the following corollary.
\begin{corollary}\label{lem:dual_cond_prob_partial_sum}
\begin{equation}\label{eq:dual_cond_prob_partial_sum}
\hat{p}_{k,n}^{(t)} = \hat{p}_{1,n}^{(t)} + (n-1) \sum_{j=1}^{k-1} \frac{1}{j} \hat{p}_{j,n-1}^{(t)},
\end{equation}
\textnormal{for all $k = 1,2,...,n$ and $t \ge 1$.}
\end{corollary}
\begin{proof} Idea: unravel equation~\eqref{eq:dual_general_recur_in_statement} of Lemma~\ref{lem:recur_k_dual_chain} step-by-step. That is, start with
\[
\hat{p}_{k,n}^{(t)} = \hat{p}_{k-1,n}^{(t)} + \frac{n-1}{k-1} \hat{p}_{k-1,n-1}^{(t)},
\]
and keep substituting into the first term on the right-hand side of the expression, each time with $k$ reduced by one. Finally, re-index the summation.
\end{proof}

\subsubsection{First Hitting Times}
The first hitting time of state $k$, $\widehat{T}_k$, of the complementary occupancy process, $\widehat{X}_t$, can be represented as a sum of independent geometrically distributed random variables, as follows
\begin{equation}\label{eq:first_hit_time_dual_chain_geo_repr}
\widehat{T}_k = G_{\frac{1}{n}} + G_{\frac{2}{n}} + \cdots  + G_{\frac{k-1}{n}},
\end{equation}
with each summand geometrically distributed with success (i.e. transition) probability $p$, i.e. $G_{p} \sim \text{Geo}(p)$, $0 < p \le 1$. From this representation, we immediately obtain its expectation and variance as:
\begin{align}\label{eq:dual_chain_1st_hit_time_expectation_and_variance}
\E[\widehat{T}_k] &= n \left( 1 + \frac{1}{2} + \cdots + \frac{1}{k-1} \right) = n \sum_{j=1}^{k-1} \frac{1}{j} = n H_{k-1}^{}\,. \\
\V[\widehat{T}_k] &= n \left( \frac{n-1}{1^2} + \frac{n-2}{2^2} + \cdots + \frac{n-k+1}{(k-1)^2} \right) = n\sum_{j=1}^{k-1} \dfrac{n-j}{j^2}.
\end{align} \label{eq:dual_chain_1st_hit_time_variance}

\subsection{The Binomial Distribution}\label{apx:binom_distribution} 
In this section we discuss the distribution that arises when we sample uniformly at random (i.e. with equiprobability) from the power set of the population $[n]$, denoted $\mathcal{P}([n])$, and defined as the collection of all subsets of $[n]$.

The universe $\mathcal{P}([n])$ has cardinality $2^n$. $\mathcal{P}([n])$ contains $\binom{n}{k}$ different $k$-subsets of $[n]$. The identity
\begin{equation}
\sum_{k=0}^{n} \binom{n}{k} = 2^n
\end{equation}
is well known (and follows directly from the Binomial Theorem). Hence, under uniform-random sampling from $\mathcal{P}([n])$, the distribution of the size of the thus obtained subset is given by
\begin{equation}\label{eq:pmf_binom}
 f(k,n) = \dfrac{\binom{n}{k}}{2^n}, \quad 0 \le k \le n.
\end{equation}

Notice that this description is essentially equivalent to a binomial distribution with $k$ ``successes" in $n$ trials, with each success having probability $p = \sfrac{1}{2}$, that is: ${f(k,n) = \binom{n}{k}(\sfrac{1}{2})^k(\sfrac{1}{2})^{n-k}}$. 

Consequently, were we to view this as a \emph{sequential} sampling process, where $X_t$ represents the number of ``successes" that have occurred upto time $t$, then the corresponding Markov chain is defined by the following $n \times n$ transition matrix
\begin{equation}\label{eq:transition_matrix_swor}
\begin{aligned}
P =  \frac{1}{2} 
\begin{bmatrix}
1 & 1              &                       &  	 					&   	        & 						&   \\
  &  1             & 1                     &  	 				    &   		    & 						&   \\
  &				   &    \small\ddots       &     \small\ddots       &   			&  						&   \\
  &                &					   &          1             &      1        &    			        &   \\
  &				   & 					   &						&		\small\ddots  & 		\small\ddots   		&   \\
  &				   &      				   &						&				&	1    	            & 1 \\
  &				   &      		           &						&				&	    		 		& 2
\end{bmatrix},
\end{aligned}
\end{equation}
and formally by
\begin{equation} 
[P]_{i,j}^{} = \frac{1}{2} \times
\begin{dcases}
1 & \text{if } j = i \neq n;  \\ 
2 & \text{if } j = i = n;     \\
1 & \text{if } j = i+1;       \\
0 & \text{otherwise}.
\end{dcases}
\end{equation}
Now note -- in analogy with \Cref{prop:pbp_ccdf_recur_matrixform} in \Cref{chpt:pbp_in_disc_time} -- that exponentiating the matrix
\begin{equation}
C =  \frac{1}{2} 
\begin{bmatrix}
2 &  1             &                       &  	 					&   	        & 						&   \\
  &  1             & 1                     &  	 				    &   		    & 						&   \\
  &				   &    \small\ddots       &     \small\ddots       &   			&  						&   \\
  &                &					   &          1             &      1        &    			        &   \\
  &				   & 					   &						&		\small\ddots  & 		\small\ddots   		&   \\
  &				   &      				   &						&				&	1    	            & 1 \\
  &				   &      		           &						&				&	    		 		& 1
\end{bmatrix},
\end{equation}
generates in its topmost row the complementary c.d.f. of the probability distribution defined in Eq. \eqref{eq:pmf_binom}. 

\subsubsection{Partial Sums of Binomial Coefficients}
Let $S(k,n) \overset{\text{def}}{=} \sum_{j=0}^k \binom{n}{j}$ denote the partial sum of the first $k+1$ binomial coefficients. And likewise, let
\begin{equation}
\overline{S}(k,n) \overset{\text{def}}{=} \sum_{j=k+1}^n \binom{n}{j} = 2^n - \sum_{j=0}^{k} \binom{n}{j}
\end{equation}
denote the complementary partial sum, that is the partial sum of the ``last" $n-k$ binomial coefficients, then $\overline{S}(k,n)$ obeys the linear recurrence relations:
\begin{equation}\label{eq:partial_compl_sum_binomial_coeff_recurrence}
\overline{S}(k,n) = \begin{dcases}
 2 \, \overline{S}(k,n-1)                       & \text{if } k = 0;     \\
\overline{S}(k,n-1) \,+\, \overline{S}(k-1,n-1) & \text{if } 1 < k < n; \\
 1 \,                                           & \text{if } k = n;  
\end{dcases}
\end{equation}
with $\overline{S}(0,0) = 2^0 = 1$.

In other words, the complementary partial sum, $\overline{S}(k,n)$, satisfies the exact same recurrence relation as its summands (the binomial coefficients), with only exception its boundary $k=0$ for $n \ge 1$. 

\begin{proof}[Proof by visual inspection.]
Recall Pascal's triangle:
\begin{equation}
\centering
\begin{tabular}{>{$n=}l<{$\hspace{12pt}}*{15}{c}}
0: &&&&&&&&1&&&&&&&\\
1: &&&&&&&1&&1&&&&&&\\
2: &&&&&&1&&2&&1&&&&&\\
3: &&&&&1&&3&&3&&1&&&&\\
4: &&&&1&&4&&6&&4&&1&&&\\
5: &&&1&&5&&10&&10&&5&&1&&\\
6: &&1&&6&&15&&20&&15&&6&&1&\\
7: &1&&7&&21&&35&&35&&21&&7&&1
\end{tabular}
\end{equation}
Now observe the first few rows of the \emph{complementary partial sum} triangle:
\begin{equation}\label{eq:partial_sum_pascal_triangle}
\centering
\begin{tabular}{>{$n=}l<{$\hspace{15pt}}*{15}{c}}
0: &&&&&&&&1&&&&&&&\\
1: &&&&&&&2&&1&&&&&&\\
2: &&&&&&4&&3&&1&&&&&\\
3: &&&&&8&&7&&4&&1&&&&\\
4: &&&&16&&15&&11&&5&&1&&&\\
5: &&&32&&31&&26&&16&&6&&1&&\\
6: &&64&&63&&57&&42&&22&&7&&1&\\
7: &128&&127&&110&&99&&64&&29&&8&&1
\end{tabular}
\end{equation}
From this latter triangle we see that each element is the sum of the ones above it, just as we have for Pascal's triangle (middle condition of \eqref{eq:partial_compl_sum_binomial_coeff_recurrence}). The left and right borders are also easily verified to be powers of two (top condition), or identical to one (bottom condition). Moreover, the difference of two adjacent terms in the same row equals the corresponding element in Pascal's triangle.
\end{proof}

When the triangle in \eqref{eq:partial_sum_pascal_triangle} is read from right to left, one finds that for the partial sum $S(k,n)$ itself a similar result holds:
\begin{equation}
S(k,n) = \begin{dcases}
 1                          & \text{if } k = 0; \\
 S(k,n-1) \,+\, S(k-1,n-1)  & \text{if } 1 < k < n; \\
 2S(k,n-1)                  & \text{if } k = n; 
\end{dcases}
\end{equation}
with $S(0,0) = 2^0 = 1$.

\subsubsection{State-Phase Duality}
Finally, let us note the form that the state-phase duality (\Cref{sec:pbp_first_hit_times}) takes for this problem. We may in this case consider a truncated variant of the Markov chain in Eq. \eqref{eq:transition_matrix_swor} with an absorbing state at any ${j = 1,2,...,n-1}$ instead of $n$. The resulting expressions for the complementary c.d.f. then become
\begin{equation}
\begin{aligned}
\Pr(X_n > k)  	  &= \frac{1}{2^n}\sum_{j=k+1}^n \binom{n}{j}  					& \qquad \text{(by definition on state space)} \\
				  &= \frac{1}{2} \sum_{j=k}^{n-1} \frac{1}{2^j}\binom{j}{k},     & \qquad \text{(as hitting time)}
\end{aligned}
\end{equation}
and equivalently for the complementary partial sums:
\begin{equation}
\overline{S}(k,n) = \sum_{j=k+1}^n \binom{n}{j} = \frac{1}{2} \sum_{j=k}^{n-1} {2^{n-j}}\binom{j}{k} .
\end{equation}

\section{Some Proofs}\label{apx:proofs}
This section contains some proofs that were left out of the main text. For completeness they are provided here.

\subsection{Deferred Proofs of Results}
Proof of \Cref{prop:pbp_cdf_recur} (\Cref{chpt:pbp_in_disc_time}).
\begin{proposition*}[CDF Recurrence]
The distribution function of any pure birth process $X_t$ induced by the vector of transition probabilities, $\mathbf{p}_{n+1}$, with $0 < p_i < 1$ ($i=0,1,...,n-1$), obeys the recurrence relations
\begin{equation}
F(k,t+1,\mathbf{p}_{n+1}) = 
\begin{dcases}
p_{k}^{} F(k-1,t,\mathbf{p}_{n+1}) + (1 - p_{k}^{}) F(k,t,\mathbf{p}_{n+1}) & \textnormal{if } \;  k < t;  \\
1 																			& \textnormal{if }  \; k \ge t;  
\end{dcases}
\end{equation}
for $1 \le k \le n$ and $t > 0$.
\end{proposition*}
\begin{proof}\label{prop:pbp_cdf_recur_proof}
Starting with the latter condition ($k \ge t$): in this case the distribution function must equal $1$ as no path has reached any state greater than $t$.

For the first condition ($k \le t$), we obtain two different expressions: one for the p.m.f., $f(k,t,\mathbf{p}_{n+1})$, in terms of the complementary c.d.f.\,; the other for the complementary c.d.f. when time is incremented by one unit. Upon combining these expressions the recurrence relation is obtained.
\begin{itemize}
\item By definition we have
\begin{equation}\label{eq:proof:pbp_pmf_def_as_ccdf_difference_apx}
f(k,t,\mathbf{p}_{n+1}) = F(k,t,\mathbf{p}_{n+1}) - F(k-1,t,\mathbf{p}_{n+1}).
\end{equation}
\item  On the other hand, we have
\begin{equation}\label{eq:proof:pbp_ccdf_time_forward_apx}
\begin{aligned}
F(k,t+1,\mathbf{p}_{n+1}) = F(k,t,\mathbf{p}_{n+1}) - p_{k}^{} f(k,t,\mathbf{p}_{n+1}), 
\end{aligned}
\end{equation}
as a consequence of relations \eqref{eq:pbp_pmf_phase_in_terms_of_stage} and \eqref{eq:pbp_cdf_phase_equiv_stage}.
\item Hence, on substituting equation \eqref{eq:proof:pbp_pmf_def_as_ccdf_difference_apx} into \eqref{eq:proof:pbp_ccdf_time_forward_apx}, we obtain
\begin{equation}
\overline{F}(k,t+1,\mathbf{p}_{n+1}) =  p_{k}^{} \overline{F}(k-1,t,\mathbf{p}_{n+1}) + (1 - p_{k}^{})\overline{F}(k,t,\mathbf{p}_{n+1}),
\end{equation}
as desired.
\end{itemize}
\end{proof}

In the same spirit we prove \Cref{prop:radomized_occ_ccdf_recur} (\Cref{sec:rand_occ_model}) here for the complementary c.d.f. of the \emph{randomized} occupancy model.
\begin{proposition*}[Complementary CDF: Recurrence]
The complementary cumulative distribution function of the randomized occupancy problem, $\overline{F}(k,t;p,n)$, obeys the recurrence relations 
\begin{equation}
\overline{F}(k,t+1;p,n) = 
\begin{dcases}
\frac{(n-k)p}{n}\overline{F}(k-1,t;p,n) + \frac{n - (n-k)p}{n} \overline{F}(k,t;p,n) & \textnormal{if } \;  k \le t;  \\
0 & \textnormal{if }  \; k > t.  
\end{dcases}
\end{equation}
for $1 \le k \le n$, and $t > 0$.
\end{proposition*}
\begin{proof}
Starting with the second case, $k > t$: it is impossible for the Markov process to reach any state greater than $t$, hence this (cumulative) probability is equal to zero.

For the other case, $k \le t$, we obtain two different expressions involving the p.m.f. in terms of the c.d.f., and substitute one into the other.

\begin{itemize}
\item We have by definition
\begin{equation}
f(k,t;p,n) = \overline{F}(k-1,t;p,n) - \overline{F}(k,t;p,n).
\end{equation}

\item On the other hand, we have
\begin{equation}
\overline{F}(k,t+1;p,n) = \overline{F}(k,t;p,n) + p_{k}^{} f(k,t;p,n).
\end{equation}

\item Substituting the first equation into the second and rearranging terms, we obtain
\begin{equation}
\begin{aligned}
\overline{F}(k,t+1;p,n) &= p_{k}^{}\overline{F}(k-1,t;p,n) + (1 - p_{k}^{}) \overline{F}(k,t;p,n) 		\\
						&= \frac{(n-k)p}{n}\overline{F}(k-1,t;p,n) + \frac{n - (n-k)p}{n} \overline{F}(k,t;p,n),  
\end{aligned}
\end{equation}
as desired.
\end{itemize}
\end{proof}

\newpage
\subsection{Other Results and Proofs}
A short collection of interesting results with proofs, which could not be included in the main text.

\subsubsection{Hitting Time Derivation of the Classical Occupancy Distribution Function}
As we saw in \Cref{sec:pbp_first_hit_times}, when deriving the distribution function of $X_t$ it is convenient to utilize its dual, that is the first hitting time distribution (in particular: relations \eqref{eq:pbp_cdf_phase_equiv_stage},  \eqref{eq:pbp_ccdf_first_hit_time_as_finite_sum_of_stage_pmf}). We shall  mimic the proof of \Cref{thm:pbp_general_ccdf} (\Cref{chpt:pbp_in_disc_time}). 

Let us recall some facts from that section in an appropriately adapted form for the classical occupancy problem. Let the process $X_t$ start in any state $0 \le r < k$, then we have for the first hitting time of state $k$ that
\begin{equation}
T_{r:k} \overset{\text{def}}{=} G_{\frac{n-r}{n}} + G_{\frac{n-1}{n}} + ... + G_{\frac{n-k-1}{n}},
\end{equation}
and
\begin{equation}
\Pr(T_{r:k{+}1} > t) = \Pr(X_t \le k \,\vert\, X_0 = r).
\end{equation}
Furthermore, we have
\begin{equation}\label{eq:occ_1st_hit_is_state}
\Pr(T_{r:k{+}1} = t) = \frac{n-k}{n} \Pr(X_{t-1} = k \,\vert\, X_0 = r).
\end{equation}
With these expressions in hand, we are ready to state the following result.
\begin{proposition}[Occupancy Distribution]\label{prop:cdf_from_geom_telescope}
The initial-state conditioned cumulative distribution function of the classical occupancy problem is given by
\begin{equation}\label{eq:cdf_from_geom_telescope}
\begin{aligned}
\Pr(X_t \le k \,\vert\, X_0 = r) &= \frac{n-k}{n^t} \binom{n-r}{n-k} \sum_{j=0}^{k-r} (-1)^{k-r-j}  \binom{k-r}{j} \frac{(r+j)^t}{n-r-j} \\
\end{aligned}
\end{equation}
for all integers $0 \le r \le k < n$, $0 \le t$; $\Pr(X_0 \le n \,\vert\, X_0 = k) = 1$ for all $0 \le k \le n$; and $0$ otherwise.
\end{proposition}

\begin{proof}
Let us briefly recall the well known manner in which an infinite geometric series can be expressed:
\begin{equation}\label{eq:occ_ancillary_geom_telescope}
\sum_{d=t+1}^\infty \lambda^d = \frac{\lambda^{t+1}}{1 - \lambda} \quad \text{for} \quad 0 \le \lambda < 1.
\end{equation}
Then starting from the premiss, we derive the desideratum as follows: 
\begin{equation}
\begin{aligned}
\Pr(X_t \le k \,\vert\, X_0 = r) 	&= \Pr(T_{r:k{+}1} > t) = \sum_{d=t+1}^\infty \Pr(T_{r:k{+}1} = d) \\
								&= \frac{n-k}{n} \sum_{d=t+1}^\infty \Pr(X_{d-1} = k \,\vert\, X_0 = r) & \text{(subst. \eqref{eq:occ_1st_hit_is_state})} \\
								&= \frac{n-k}{n} \sum_{d=t}^\infty \binom{n-r}{n-k} \sum_{j=0}^{k-r} (-1)^{k-r-j} \binom{k-r}{j} \left(\frac{j+r}{n}\right)^d & \text{(subst. \eqref{eq:explicit_from_eigendecomp})}   \\ 
								&=  \frac{n-k}{n} \binom{n-r}{n-k} \sum_{j=0}^{k-r} (-1)^{k-r-j} \binom{k-r}{j} \sum_{d=t}^\infty \left(\frac{j+r}{n}\right)^d &  \text{(reorder summation)} \\
								&=  \frac{n-k}{\cancel{n}} \binom{n-r}{n-k} \sum_{j=0}^{k-r} (-1)^{k-r-j} \binom{k-r}{j} \frac{\cancel{n}}{n^{t}} \frac{(r+j)^{t}}{n-r-j} & \text{(use \eqref{eq:occ_ancillary_geom_telescope})} \\
								&=  \frac{n-k}{n^t} \binom{n-r}{n-k} \sum_{j=0}^{k-r} (-1)^{k-r-j} \binom{k-r}{j} \frac{(r+j)^{t}}{n-r-j}.
\end{aligned}
\end{equation}
\end{proof}

Consequently, an explicit formula for the unconditional cumulative occupancy distribution is obtained as a special case by taking $r=0$:
\begin{equation}\label{eq:cdf_explicit}
F(k,t,n) = \frac{1}{n^t} \sum_{j=0}^{k} (-1)^{k-j} \binom{n-j-1}{n-k} \binom{n}{j} j^t \\ 
		 = \frac{(n)_{k+1}^{}}{n^t}  \frac{1}{k!} \sum_{j=0}^{k} (-1)^{k-j} \binom{k}{j} \frac{j^t}{n-j}.
\end{equation}

\Cref{prop:cdf_from_geom_telescope} thus generalizes the known expression for the c.d.f. \citep[cf.][]{feller1968ed3} to all initial-state conditioned distribution functions, $\Pr(X_t \le k \, \vert \, X_0 = r)$. Recalling expression \eqref{eq:def_fin_diff} from the \nameref{chpt:preliminaries}, we are able to compactly express the c.d.f. as a $k$th order finite difference,
\begin{equation}\label{eq:cdf_simple_finite_difference}
F(k,t,n) = \frac{(n)_{k+1}}{n^t} \frac{1}{k!} \Delta^k \left[ \frac{x^t}{n-x}  \right]_{x=0}^{},
\end{equation}
underscoring the simplicity of expression \eqref{eq:cdf_from_geom_telescope}. And more generally, 
\begin{equation}
\Pr(X_t = k \, \vert \, X_0 = r) = \frac{(n-r)_{k-r+1}}{n^t} \frac{1}{(k-r)!} \Delta^{k-r} \left[ \frac{x^t}{n-x}  \right]_{x=r}^{}.
\end{equation}

\subsubsection{CDF in terms of Conditionals}
Perhaps one of the most interesting proofs that is not included in the main text, is the following. The result is notable and its derivation a handy algebraic manipulation.

Recall that the cumulative distribution function $F$ is defined as the sum of probability masses over states $j = 0,1,2,...,k$, i.e.
\begin{equation}
\begin{aligned}
F(k,t,n) &= \sum_{j=0}^k p_{0,j}^{(t)} = \sum_{j=0}^k f(j,n,t) = \frac{1}{n^t} \sum_{j=0}^k (n)_j \stirlingtwo{t}{j}.
\end{aligned}		 
\end{equation}
From matrix manipulation of the eigendecomposition we are able to obtain an equivalent expression for $F$ as the sum of initial-state conditioned probability masses of a classical occupancy distributed random variable \emph{over a smaller population}, namely $[n-1]=\{1,2,...,n-1\}$.

First, we need the following auxiliary result.
\begin{corollary}\label{cor:SigmaInv_times_U}
The inverse of $U_n^{-1} \Sigma$ is given by
\begin{equation}
\Sigma_n^{-1} U_n^{} = 
\left[ \begin{array}{c|c}
		U_{n-1} 	& \mathbf{0} \\
		\hline 
		\mathbf{0'} & 1
\end{array} \right]
\end{equation}
\end{corollary}
\begin{proof}
$(U_n^{} \Sigma) A = I_n^{}$ if and only if $A = (U_n^{} \Sigma)^{-1} = \Sigma^{-1} U_n^{}$.
\end{proof}
This is a direct corollary of \Cref{lem:Uinv_times_Sigma}. Summarizing, we have:
\begin{equation}\label{eq:Uinv_Sigma_and_SigmaInv_U} 
U_{n}^{-1} \Sigma = 
\left[
\begin{array}{c|c}
U_{n-1}^{-1} & \mathbf{0}  \\
\hline  				   
\mathbf{0'}  & 1 
\end{array}
\right], \quad
\Sigma^{-1} U_{n} = 
\left[
\begin{array}{c|c}
U_{n-1} & \mathbf{0}  \\
\hline  				   
\mathbf{0'}  & 1 
\end{array}
\right].
\end{equation}
And therefore, 
\begin{equation}\label{eq:Sigma_times_transition_matrix}
\begin{aligned}
P_{n}^t\Sigma &= U_{n}^{} \Lambda_{n}^t U_{n}^{-1} \Sigma \\
                &= \Big( \Sigma \Sigma^{-1} \Big) U_{n} \Lambda_{n}^t U_{n}^{-1} \Sigma  & \text{(identity matrix)}\\
                &= \Sigma \Big( \Sigma^{-1} U_{n}^{} \Big) \Lambda_{n}^t \Big( U_{n}^{-1} \Sigma \Big)  \\
                &= \Sigma  \left[
									\begin{array}{c|c}
									U_{n-1}^{-1} & \mathbf{0}  \\
									\hline  				   
									\mathbf{0'}  & 1 
									\end{array}
							\right]
							\left[
									\begin{array}{c|c}
									\Lambda_{n-1}^{-1} & \mathbf{0}  \\
									\hline  				   
									\mathbf{0'}  & 1 
									\end{array}
							\right]													
							\left[
									\begin{array}{c|c}
									U_{n-1} & \mathbf{0}  \\
									\hline  				   
									\mathbf{0'}  & 1 
									\end{array}
							\right] & \text{(by Eq. \eqref{eq:Uinv_Sigma_and_SigmaInv_U})} \\
                &=  \Sigma
					\left[
						\begin{array}{c|c}
						 U_{n-1} \Lambda_{n-1}^t U_{n-1}^{-1} & \mathbf{0}  \\
						\hline  				   
						\mathbf{0}'  & 1 
						\end{array}
					\right]  \\
                &= \frac{(n-1)^t}{n^t} \Sigma
					\left[
						\begin{array}{c|c}
						 U_{n-1} \frac{n^t}{(n-1)^t} \Lambda_{n-1}^t U_{n-1}^{-1} & \mathbf{0}  \\
						\hline  				   
						\mathbf{0}'  & \frac{n^t}{(n-1)^t} 
						\end{array}
					\right] & \text{(multiplying by 1)} \\
			    &= 	\frac{(n-1)^t}{n^t} \Sigma
					\left[
						\begin{array}{c|c}
						P_{n-1}^t & \mathbf{0}  \\
						\hline  				   
						\mathbf{0}'  & \frac{n^t}{(n-1)^t} 
						\end{array}
					\right] \\		
				&= 	\left[ \begin{array}{c|c} (1-\frac{1}{n})^t \Sigma P_{n-1}^t & \mathbf{1}  \end{array} \right].						
\end{aligned}
\end{equation}
The \textit{pointe} is that $\left(  \Sigma^{-1} U \right) \Lambda  \left( U^{-1} \Sigma \right) $  is a the transition matrix of a reduced (smaller population) Markov chain, $\widetilde{P}$, on which left-multiplication by $\Sigma$ has the effect of taking the cumulative sum column-wise from bottom to top. Hence, the following proposition.

\begin{proposition}\label{prop:clas_occ_cdf_as_sum_cond_probs_smaller_pop}
Let $X_t$ and  $\widetilde{X}_t$ denote random variables with classical occupancy distributions over populations $[n]=\{1,2,...,n\}$ and $[n-1]=\{1,2,...,n-1\}$, respectively. Then the cumulative distribution function of $X_t$ equals
\begin{equation}
\Pr(X_t \le k) = ( 1 - \frac{1}{n} )^t \sum_{j=0}^k  \Pr(\widetilde{X}_t = k \,\vert\, \widetilde{X}_0 = j),
\end{equation}
for all $1 \le k \le n-1$ and $t > 0$.
\end{proposition}
\begin{proof}
See foregoing derivation \eqref{eq:Sigma_times_transition_matrix} and consider the top row.
\end{proof}

\section{Concrete Example of a Path}
Denote a realization of the process $\mathbf{X}$ as the tuple, or vector: $[X_0, X_1, X_2, ..., X_{d}]^\top$. For example, a realization of this process, denoted $\mathbf{x}_t$, may look like
\begin{equation*}
\mathbf{x_{11}} = [0,1,2,2,3,4,5,5,5,6,7,8]^\top.
\end{equation*}
This is just one possible realization of a process $\mathbf{X}$ which transitioned in $t=11$ steps from state $0$ to state $8$ with one stop in state $2$ and two stops in state $5$, and zero stops in all other states. If the population (or alphabet) would equal $[n] = \{1,2,3,4,...,10\}$, the probability of this path would be equal to
\begin{equation*}
\Pr(\mathbf{X} = \mathbf{x_{11}}) = \frac{n (n-1) \cdot  2 \cdot (n-2)(n-3)(n-4) \cdot 5  \cdot 5 \cdot(n-5)(n-6)(n-7)(n-8)}{n^{11}},
\end{equation*}
with $n=10$. It is immediately clear that each $t$-step $(0 \to k)$-path always needs to contain all transitions between $0$ and $k$. There is however, ``freedom" in the choice of stops over states $\{1,2,...,k\}$ when $t > k$. It is exactly there that Stirling numbers of the second kind come into play.

The dual representation of path $\mathbf{x_{11}}$, is as a collection of the waiting times per state (excluding the last), concretely
\begin{equation}
\mathbf{w_{0:7}} = [0,0,1,0,0,2,0,0]^\top,
\end{equation}
that represent the time spent in each state.

\section{Concrete Matrix Examples}\label{apx:matrix_examples}

\subsection{Sum and Difference Matrices}
A $10 \times 10$ sum matrix:
\begin{equation}
\Sigma = \\
\begin{bmatrix*}[r]
1 & 1 	& 1 	& 1 	& 1 	& 1 	& 1  	& 1 	& 1  & 1  \\
0 & 1 	& 1 	& 1 	& 1 	& 1 	& 1  	& 1 	& 1  & 1  \\
0 & 0  	& 1 	& 1 	& 1 	& 1 	& 1  	& 1 	& 1  & 1  \\
0 & 0  	& 0  	& 1 	& 1 	& 1 	& 1  	& 1 	& 1  & 1  \\
0 & 0  	& 0  	& 0   	& 1 	& 1	 	& 1  	& 1 	& 1  & 1  \\
0 & 0  	& 0  	& 0   	& 0     & 1  	& 1 	& 1  	& 1  & 1  \\
0 & 0  	& 0  	& 0   	& 0   	& 0   	& 1  	& 1 	& 1  & 1  \\
0 & 0  	& 0  	& 0   	& 0   	& 0    	& 0   	& 1 	& 1  & 1  \\
0 & 0  	& 0  	& 0   	& 0   	& 0    	& 0   	& 0   	& 1  & 1  \\
0 & 0  	& 0  	& 0   	& 0   	& 0    	& 0   	& 0   	& 0  & 1 
\end{bmatrix*}. 
\end{equation}

Its square:
\begin{equation}
\Sigma^2 = \\
\begin{bmatrix*}[r]
\phantom{0}1 & \phantom{0}2 & \phantom{0}3 & \phantom{0}4 & \phantom{0}5 & \phantom{0}6 & \phantom{0}7 & \phantom{0}8 & \phantom{0}9 & 10 \\
0 & 1 & 2 & 3 & 4 & 5 & 6 & 7 & 8 & 9  \\
0 & 0 & 1 & 2 & 3 & 4 & 5 & 6 & 7 & 8  \\
0 & 0 & 0 & 1 & 2 & 3 & 4 & 5 & 6 & 7  \\
0 & 0 & 0 & 0 & 1 & 2 & 3 & 4 & 5 & 6  \\
0 & 0 & 0 & 0 & 0 & 1 & 2 & 3 & 4 & 5  \\
0 & 0 & 0 & 0 & 0 & 0 & 1 & 2 & 3 & 4  \\
0 & 0 & 0 & 0 & 0 & 0 & 0 & 1 & 2 & 3  \\
0 & 0 & 0 & 0 & 0 & 0 & 0 & 0 & 1 & 2  \\
0 & 0 & 0 & 0 & 0 & 0 & 0 & 0 & 0 & 1 
\end{bmatrix*}. 
\end{equation}

Its inverse, a $10 \times 10$ difference matrix:
\begin{equation}
\Sigma^{-1} = \\
\begin{bmatrix*}[r]
1 & -1 & 0 	& 0 	& 0 	& 0 	& 0  	& 0 	& 0  	& 0 	\\
0 & 1  & -1 & 0  	& 0 	& 0   	& 0 	& 0  	& 0 	& 0  	\\
0 & 0  & 1  & -1  	& 0  	& 0  	& 0 	& 0 	& 0  	& 0 	\\
0 & 0  & 0  & 1   	& -1  	& 0  	& 0 	& 0  	& 0 	& 0  	\\
0 & 0  & 0  & 0   	& 1   	& -1   	& 0  	& 0 	& 0  	& 0 	\\
0 & 0  & 0  & 0   	& 0   	& 1    	& -1  	& 0   	& 0 	& 0  	\\
0 & 0  & 0  & 0   	& 0   	& 0    	& 1   	& -1  	& 0  	& 0 	\\
0 & 0  & 0  & 0   	& 0   	& 0    	& 0   	& 1   	& -1 	& 0  	\\
0 & 0  & 0  & 0   	& 0   	& 0    	& 0   	& 0   	& 1  	& -1 	\\
0 & 0  & 0  & 0   	& 0   	& 0    	& 0   	& 0   	& 0  	&  1 
\end{bmatrix*}. 
\end{equation}

\subsection{Transition Matrix and Related Matrices}\label{apx:matrix_examples_transmat}
A $10 \times 10$ transition matrix $P$ corresponding to the classical occupancy chain:
\begin{equation}
P = \dfrac{1}{9} \begin{bmatrix*}[c]
0 & 9 & 0 & 0 & 0 & 0 & 0 & 0 & 0 & 0  	\\
0 & 1 & 8 & 0 & 0 & 0 & 0 & 0 & 0 & 0   \\
0 & 0 & 2 & 7 & 0 & 0 & 0 & 0 & 0 & 0   \\
0 & 0 & 0 & 3 & 6 & 0 & 0 & 0 & 0 & 0   \\
0 & 0 & 0 & 0 & 4 & 5 & 0 & 0 & 0 & 0   \\
0 & 0 & 0 & 0 & 0 & 5 & 4 & 0 & 0 & 0   \\
0 & 0 & 0 & 0 & 0 & 0 & 6 & 3 & 0 & 0   \\
0 & 0 & 0 & 0 & 0 & 0 & 0 & 7 & 2 & 0   \\
0 & 0 & 0 & 0 & 0 & 0 & 0 & 0 & 8 & 1   \\
0 & 0 & 0 & 0 & 0 & 0 & 0 & 0 & 0 & 9   \\
\end{bmatrix*}.
\end{equation}

Examples of a $10 \times 10$ upper-triangular, downward-pointing Pascal matrix and its inverse:
\begin{equation}\label{eq:apx_U_binomcoeff}
U\phantom{^{-1} \Sigma} = \\
\begin{bmatrix*}[c]
\binom{9}{0} & \binom{9}{1} & \binom{9}{2} & \binom{9}{3} & \binom{9}{4} & \binom{9}{5} & \binom{9}{6} & \binom{9}{7} & \binom{9}{8} & \binom{9}{9} \\
0 & \binom{8}{0} & \binom{8}{1}  & \binom{8}{2} & \binom{8}{3}  & \binom{8}{4}  & \binom{8}{5} & \binom{8}{6} & \binom{8}{7} & \binom{8}{8} \\
0 & 0 & \binom{7}{0}  & \binom{7}{1}  & \binom{7}{2}  & \binom{7}{3}  & \binom{7}{4} & \binom{7}{5} & \binom{7}{6} & \binom{7}{7} \\
0 & 0 & 0  & \binom{6}{0}  & \binom{6}{1}   & \binom{6}{2}  & \binom{6}{3} & \binom{6}{4} & \binom{6}{5} & \binom{6}{6} \\
0 & 0 & 0  & 0  & \binom{5}{0}   & \binom{5}{1}   & \binom{5}{2} & \binom{5}{3} & \binom{5}{4} & \binom{5}{5} \\
0 & 0 & 0  & 0  & 0   & \binom{4}{0}   & \binom{4}{1}  & \binom{4}{2}  & \binom{4}{3} & \binom{4}{4} \\
0 & 0 & 0  & 0  & 0   & 0   & \binom{3}{0}  & \binom{3}{1}  & \binom{3}{2} & \binom{3}{3} \\
0 & 0 & 0  & 0  & 0   & 0   & 0  & \binom{2}{0}  & \binom{2}{1} & \binom{2}{2} \\
0 & 0 & 0  & 0  & 0   & 0   & 0  & 0  & \binom{1}{0} & \binom{1}{1} \\
0 & 0 & 0  & 0  & 0   & 0   & 0  & 0  & 0 & \binom{0}{0}
\end{bmatrix*},
\end{equation}

\begin{equation}\label{eq:apx_U_binomcoeff_explicit}
U\phantom{^{-1}\Sigma} = \\
\begin{bmatrix*}[r]
1 & \phantom{-}9 & 36 & \phantom{-}84 & 126 & \phantom{-}126 & 84 & \phantom{-}36 & \phantom{-}9 & \phantom{-}1 \\
0 & 1 & \phantom{-}8  & 28 & \phantom{-}56  & 70  & \phantom{-}56 & 28 & 8 & 1 \\
0 & 0 & 1  & 7  & 21  & 35  & 35 & 21 & 7 & 1 \\
0 & 0 & 0  & 1  & 6   & 15  & 20 & 15 & 6 & 1 \\
0 & 0 & 0  & 0  & 1   & 5   & 10 & 10 & 5 & 1 \\
0 & 0 & 0  & 0  & 0   & 1   & 4  & 6  & 4 & 1 \\
0 & 0 & 0  & 0  & 0   & 0   & 1  & 3  & 3 & 1 \\
0 & 0 & 0  & 0  & 0   & 0   & 0  & 1  & 2 & 1 \\
0 & 0 & 0  & 0  & 0   & 0   & 0  & 0  & 1 & 1 \\
0 & 0 & 0  & 0  & 0   & 0   & 0  & 0  & 0 & 1
\end{bmatrix*},
\end{equation}

\begin{equation}\label{eq:apx_Uinv_binomcoeff_explicit}
U^{-1}\phantom{\Sigma} = \\
\begin{bmatrix*}[r]
1 & -9 & 36 & -84 & 126 & -126 & 84  & -36 & 9  & -1 	\\
0 & 1  & -8 & 28  & -56 & 70   & -56 & 28  & -8 &  1	\\
0 & 0  & 1  & -7  & 21  & -35  & 35  & -21 & 7  & -1 	\\
0 & 0  & 0  & 1   & -6  & 15   & -20 & 15  & -6 &  1  	\\
0 & 0  & 0  & 0   & 1   & -5   & 10  & -10 & 5  & -1 	\\
0 & 0  & 0  & 0   & 0   & 1    & -4  & 6   & -4 &  1  	\\
0 & 0  & 0  & 0   & 0   & 0    & 1   & -3  & 3  & -1 	\\
0 & 0  & 0  & 0   & 0   & 0    & 0   & 1   & -2 &  1  	\\
0 & 0  & 0  & 0   & 0   & 0    & 0   & 0   & 1  & -1 	\\
0 & 0  & 0  & 0   & 0   & 0    & 0   & 0   & 0  &  1 
\end{bmatrix*},
\end{equation}

\begin{equation}\label{eq:apx_UinvSigma_binomcoeff_explicit}
U^{-1}\Sigma = \\
\begin{bmatrix*}[r]
1 & -8 & 28 & -56 	&  70  	& -56 	&  28  	&  8  &  1  &  0  \\
0 & 1  & -7 & 21  	& -35 	& 35  	& -21   &  7  & -1  &  0  \\
0 & 0  & 1  & -6  	& 15  	& -20 	&  15   & -6  &  1  &  0  \\
0 & 0  & 0  & 1   	& -5  	& 10  	& -10   &  5  & -1  &  0  \\
0 & 0  & 0  & 0   	& 1   	& -4  	&   6   & -4  &  1  &  0  \\
0 & 0  & 0  & 0   	& 0   	& 1   	&  -3   &  3  & -1  &  0  \\
0 & 0  & 0  & 0   	& 0   	& 0   	&   1   & -2  &  1  &  0  \\
0 & 0  & 0  & 0   	& 0   	& 0   	&   0   &  1  & -1  &  0  \\
0 & 0  & 0  & 0   	& 0   	& 0   	&   0   &  0  &  1  &  0  \\
0 &	\phantom{-}0&	\phantom{-}0&	\phantom{-3}0&	\phantom{-8}0&	\phantom{-1}0&	\phantom{-1}0&	\phantom{-8}0&	\phantom{-3}0&	\phantom{-}1 
\end{bmatrix*}.
\end{equation}

\newpage
\subsection{Cumulative Distribution Matrix and Related Matrices}\label{apx:matrix_examples_ccdf_mat}
A $9 \times 9$ complementary cumulative distribution matrix, $C$ (cf. \Cref{sec:cocc_matrix_implementation}, Eq. \eqref{eq:ccdf_recur_matrixform}):
\begin{equation}
C = \dfrac{1}{9}\begin{bmatrix*}[c]
9  & 8 & 0 & 0 & 0 & 0 & 0 & 0 & 0 \\
0  & 1 & 7 & 0 & 0 & 0 & 0 & 0 & 0 \\
0  & 0 & 2 & 6 & 0 & 0 & 0 & 0 & 0 \\
0  & 0 & 0 & 3 & 5 & 0 & 0 & 0 & 0 \\
0  & 0 & 0 & 0 & 4 & 4 & 0 & 0 & 0 \\
0  & 0 & 0 & 0 & 0 & 5 & 3 & 0 & 0 \\
0  & 0 & 0 & 0 & 0 & 0 & 6 & 2 & 0 \\
0  & 0 & 0 & 0 & 0 & 0 & 0 & 7 & 1 \\
0  & 0 & 0 & 0 & 0 & 0 & 0 & 0 & 8 \\
\end{bmatrix*}.
\end{equation}

\noindent The matrices of eigenvectors of $C$:
\begin{equation}
V \phantom{^{-1}} = \begin{bmatrix*}[r]
1 & -1 & -8 & -28 & -56 & -70 & -56 & -28 & -8 \\
0 & 1  & 7  & 21  & 35  & 35  & 21  & 7   & 1  \\
0 & 0  & 1  & 6   & 15  & 20  & 15  & 6   & 1  \\
0 & 0  & 0  & 1   & 5   & 10  & 10  & 5   & 1  \\
0 & 0  & 0  & 0   & 1   & 4   & 6   & 4   & 1  \\
0 & 0  & 0  & 0   & 0   & 1   & 3   & 3   & 1  \\
0 & 0  & 0  & 0   & 0   & 0   & 1   & 2   & 1  \\
0 & 0  & 0  & 0   & 0   & 0   & 0   & 1   & 1  \\
0 & 0  & 0  & 0   & 0   & 0   & 0   & 0   & 1 
\end{bmatrix*},
\end{equation}

\begin{equation}
V^{-1} = \begin{bmatrix*}[r]
1 & 1  & 1  & 1   & 1   & 1   & 1   & 1  & 1  \\
0 & 1  & -7 & 21  & -35 & 35  & -21 & 7  & -1 \\
0 & 0  & 1  & -6  & 15  & -20 & 15  & -6 & 1  \\
0 & 0  & 0  & 1   & -5  & 10  & -10 & 5  & -1 \\
0 & 0  & 0  & 0   & 1   & -4  & 6   & -4 & 1  \\
0 & 0  & 0  & 0   & 0   & 1   & -3  & 3  & -1 \\
0 & 0  & 0  & 0   & 0   & 0   & 1   & -2 & 1  \\
0 & 0  & 0  & 0   & 0   & 0   & 0   & 1  & -1 \\
0 &	\phantom{-}0  & 0  & \phantom{-2}0   &	0   & 0   &	 0  & \phantom{-2}0  &	1 
\end{bmatrix*}.
\end{equation}


\end{document}